\documentclass[a4paper, 11pt,reqno]{amsart}

\usepackage{amsmath}
\usepackage{amssymb}
\usepackage{hyperref}
\usepackage{theoremref} 
\usepackage{amsfonts}
\usepackage[dvipsnames]{xcolor}
\usepackage{mathtools}
\usepackage{amsthm}
\usepackage{color}
\usepackage{ulem}
\usepackage[utf8]{inputenc}
\usepackage{array}
\usepackage{enumitem}
\usepackage[english]{babel}
\usepackage{tikz}
\usetikzlibrary{automata, positioning, arrows}

\calclayout

\setlist{topsep=0mm,partopsep=0mm,itemsep=1mm}

\theoremstyle{plain}
\newtheorem{propn}{Proposition}[section]
\newtheorem{thm}[propn]{Theorem}

\newtheorem{lemma}[propn]{Lemma}
\newtheorem{cor}[propn]{Corollary}
\theoremstyle{definition}
\newtheorem{exam}[propn]{Example}

\theoremstyle{definition}
\newtheorem{que}[propn]{Question}
\theoremstyle{definition}
\newtheorem{defn}[propn]{Definition}

\theoremstyle{remark}

\newcommand{\N}{\mathbb{N}}

\newcommand{\R}{\mathcal{R}}
\renewcommand{\L}{\mathcal{L}}
\renewcommand{\H}{\mathcal{H}}
\newcommand{\J}{\mathcal{J}}
\newcommand{\K}{\mathcal{K}}

\newcommand{\epsa}{\varepsilon_{A}}
\newcommand{\epsb}{\varepsilon_{B}}
\newcommand{\epsc}{\varepsilon_{C}}

\newcommand{\fibpro}{\Pi(\varphi,\psi)}

\tikzset{
->, 
>=stealth', 
node distance=3cm, 
every state/.style={thick, fill=gray!10,circle, draw, minimum size=1.2cm}, 
initial text=$ $, 
}

\begin{document}

\title[Finitary properties for fiber products]{On finitary properties for fiber products of free semigroups and free monoids}
\author[A. Clayton]{Ashley Clayton}
\address{School of Mathematics and Statistics, University of St Andrews, St Andrews, Scotland, UK}
\email{$\{$ac323$\}$@st-andrews.ac.uk}


\begin{abstract}
We consider necessary and sufficient conditions for finite generation and finite presentability for fiber products of free semigroups and free monoids. We give a necessary and sufficient condition on finite fiber quotients for a fiber
product of two free monoids to be finitely generated, and show that all such fiber products are also finitely presented.
By way of contrast, we show that fiber products of free semigroups over finite fiber quotients are never finitely generated.  We then consider fiber products of free semigroups over infinite semigroups, and show that for such a fiber product to be finitely generated, the quotient must be infinite but finitely generated, idempotent-free, and $\mathcal{J}$-trivial. Finally, we construct automata accepting the indecomposable elements of the fiber product of two free monoids/semigroups over free monoid/semigroup fibers, and give a necessary and sufficient condition for such a product to be finitely generated.
\end{abstract}

\maketitle

\section{Introduction}
A \textit{subdirect product} of two algebras $A$ and $B$ is a subalgebra of the direct product, for which the natural projections onto $A$ and $B$ are surjective. In particular, the direct product of two algebras is a subdirect product, for which finitary properties have been well studied for groups. Most results here indicate that direct products of groups have a well behaved structure based on their constituent factors. That is, two groups $G$ and $H$ have the following properties (amongst others) if and only if $G \times  H$ also does: finitely generated; finitely presented; residually finite; nilpotent; solvable; and having decidable word problem. \\

By way of contrast, subdirect products of groups have more complicated behaviour in general, which has been particularly well exhibited for subdirect products of free groups. There are examples (stemming from \cite[Theorem 1]{baumslaugandroseblade}) which are not finitely generated \cite[Example 3]{bridsonmiller}; finitely generated without being finitely presented \cite{Grunewald78}; and finitely generated but with undecidable membership problem \cite{Mihailova66}. Describing their substructure complexity, any two non-abelian free groups $G$ and $H$ have uncountably many pairwise non-isomorphic subdirect products of $G$ and $H$ \cite[Corollary B]{bridsonmiller}.\\

By a result due to Goursat \cite{goursat}, subdirect products of groups arise as fiber products and vice versa, and are hence constructible in some sense. Comparatively by a more recent result due to Fleischer \cite{fleisch}, this is also true more generally for varieties of algebras which are congruence permutable (that is, all congruences commute with each other under composition) and equivalently varieties whose languages contain Mal'cev terms, which include the varieties of groups, rings and Lie algebras. In varieties whose algebras do not contain Mal'cev terms however, fiber products are subdirect, but not conversely. It is natural to investigate the boundary between those subdirect products which are constructible via fiber products, and those which are not in such varieties. \\

Further, the setting of the subdirect product structure in Universal Algebra owes itself to many natural questions relating to generation, presentation, and decidability for given varieties. The varieties of semigroups, monoids and lattices are not congruence permutable, and recent results indicate as for groups that subdirect products of the free objects in these varieties are already interesting. For example for the free monogenic semigroup viewed as $\N$, there are uncountably many pairwise non-isomorphic subdirect products of $\N^{k}$ for any $k \geq 2$ \cite{acnr}. For questions of finite generation, Mayr and Ru\v{s}kuc \cite{pmnr} have given some examples of the complications arising for free monoids: there exist fiber products of two free monogenic monoids over a finite fiber quotient which are not finitely generated \cite[Example 7.1]{pmnr}, and projection onto several factors is not sufficient for finite generation of subdirect products of more than two monoids \cite[Example 7.3]{pmnr}. Following this, they ask the below question:

\begin{que}[\cite{pmnr}, Problem 7.2]\thlabel{questionfromnik} Find necessary and sufficient conditions for a fiber product of
finitely generated monoids over a finite monoid to be finitely generated. More
specifically, is it decidable whether a fiber product of two finitely generated free
monoids over a finite quotient is finitely generated?
\end{que}

Following this preceeding work, in this paper we undertake an investigation into finite generation and presentation for fiber products of free semigroups and monoids. In Section \ref{prelims}, we introduce the necessary preliminary materials concerning subdirect products, fiber products, and formal language and automata theory used. In Section \ref{sec2}, we consider fiber products of free semigroups and monoids over finite fiber quotients, answering \thref{questionfromnik} in this case. In particular, we give the following results:
\begin{itemize}
\item there are no finitely generated fiber products of two free semigroups over a finite fiber quotient (\thref{FreeSemiFin});
\item it is decidable whether a fiber product of two finitely generated free monoids over a finite fiber quotient is finitely generated, and give necessary and sufficient conditions on the fiber quotient (\thref{FreeMonIff});
\item if a fiber product of two finitely generated free monoids over a finite fiber quotient is finitely generated, then it is also finitely presented (\thref{fingenifffinpres}), and give a presentation in this case.
\end{itemize}

In Section \ref{sec3}, we consider necessary conditions for finite generation for fiber products of free semigroups and monoids over infinite fiber quotients. In particular, we show the following:
\begin{itemize}
\item finitely generated fiber products of free semigroups have finitely generated, \linebreak $\mathcal{J}$-trivial, idempotent-free fiber quotients (\thref{fgcompsandfiberlemma}, \thref{torsionfree}, \thref{jtrivfiber});
\item a fiber product of free semigroups with fiber quotient $\N$ is finitely generated if and only if at least one of the epimorphisms is a constant map (\thref{Ntheorem});
\item fiber products of free semigroups over non-monogenic free commutative semigroup fiber quotients are not finitely generated (\thref{freecommsgp}).
\end{itemize}

In Section \ref{decsec}, we consider decision problems on fiber products of free semigroups and monoids with free fiber quotients, showing the following:
\begin{itemize}
\item the generalised word problem for a fiber product of semigroups in the direct product is decidable if and only if the word problem of the fiber quotient is decidable;
\item given a fiber product of two free semigroups (monoids) over a free semigroup (monoid) fiber quotient, one can ask whether or not it is finitely generated. This finite generation problem is decidable, for which we construct suitable finite state automata (\thref{fgnocyclesaut}, \thref{A'nocyclesfg}).
\end{itemize}

In Section \ref{nonfibsec}, we make some remarks on the number of finitely generated subdirect products of two free semigroups $A^{+}$ and $B^{+}$, which are generated by some subset of $A \times B$. In particular, we count the number of such subdirect products (\thref{numberoffgsubs}) as well as the number of fiber products (\thref{numberoffgfibs}), and make some remarks on their sparsity in $A^{+} \times B^{+}$.

Finally, we conclude in Section \ref{quesec} with some arising open questions.

\section{Preliminaries}\label{prelims}
Throughout, a \textit{subdirect product of semigroups (monoids)} $S$ \text{and} $T$ is a subsemigroup (submonoid) $U$ of $S\times T$ such that the projection maps
\[
\pi_{S} : U \rightarrow S, \,
\pi_{T} : U \rightarrow T
\]
are surjections. In this case, we write $U \leq_{\textup{sd}} S\times T$. This definition naturally extends to a \textit{subdirect product of a family of semigroups (monoids)} $\{S_{i}\}_{i \in I}$, being a subsemigroup (submonoid) $U$ of the direct product $\prod_{i \in I} S_{i}$ for which each of the projection maps $\pi_{i} : U \rightarrow S_{i}$ are surjections. For this paper however, we only consider finite families.

If $\varphi: S \rightarrow F$, $\psi : T \rightarrow F$ are two epimorphisms onto a common semigroup (monoid) $F$, then the \textit{fiber product of} $S$ \textit{and} $T$ \textit{with respect to} $\varphi, \psi$ is the subdirect product of $S\times T$ given by the set
\[
\{ (s,t) \in S\times T : \varphi(s) = \psi(t) \},
\]
with multiplication inherited from $S\times T$. We will write $\fibpro$ to denote the fiber product. The semigroup $F$ is called the \textit{fiber quotient} of the fiber product. Similarly to subdirect products, fiber products can be defined on families of semigroups and monoids as well. Fiber products are indeed subdirect products of $S$ and $T$, but not all subdirect products can be obtained in this way. The following result classifies when the two notions coincide:

\begin{lemma}[cf. Fleischer's Lemma \cite{fleisch}, Lemma 10.1] \thlabel{FleischerLemma} Let $S$, $T$ be semigroups, let $U \leq_{\textup{sd}} S \times T$. Define
\[\pi_{S} : U \rightarrow S := (s,t) \mapsto s,\]
\[\pi_{T} : U \rightarrow T := (s,t) \mapsto t.\]
 Then $U$ is a fiber product if and only if the kernel congruences of the above projection maps commute under composition, that is
\[\textup{ker}\,\pi_{S} \circ \textup{ker}\,\pi_{T} = \textup{ker}\,\pi_{T} \circ \textup{ker}\,\pi_{S}.\] \end{lemma}

An \textit{alphabet} is a set $A$ consisting of formal symbols, where the elements of $A$ are referred to as  \textit{letters}. The \textit{free semigroup} $A^{+}$ is the set of all finite non-empty strings of letters over $A$, with the operation of concatenation of strings. Allowing for the empty string $\varepsilon$ (being the string consisting of no letters), the \textit{free monoid} $A^{*}$ is the set of all finite words over $A$, again with the operation of concatenation. A \textit{word over} $A$ is an element of $A^{*}$ (where the \textit{empty word} is identified to be the empty string $\varepsilon$). For a word $w \in A^{+}$, we will write $w_{i}$ for the $i$-th letter of $w$. A \textit{prefix} of a word $w \in A^{*}$ is an element $u \in A^{*}$ such that there exists $v \in A^{*}$ with $w = uv$. A \textit{proper prefix} $u$ of $w \in A^{*}$ is a prefix which is not equal to $w$. If $u$ is a prefix of $w$, we will write $u \leq_{p} w$ (and $u <_{p} w$ if $u$ is a proper prefix). Similarly, a \textit{suffix} of a word $w \in A^{*}$ is an element $v \in A^{*}$ such that  there exists $u \in A^{*}$ with $w = uv$, and a \textit{proper suffix} of $w \in A^{*}$ is a suffix $v$ not equal to $w$. If $v$ is a suffix of $w$, we will write $v \leq_{s} w$ (and $v <_{s} w$ if $v$ is a proper suffix). For a prefix $u$ (resp. suffix $v$) of $w$, we write $u^{-1}w$ (resp. $wv^{-1}$) to mean the unique word $u'$ (resp. $v'$) such that $uu' = w$ (resp. $v'v = w$), or equivalently the word $w$ with prefix $u$ (resp. suffix $v$) removed.

For a semigroup $S$, an \textit{idempotent} is an element $e \in S$ such that $e^{2} = e$. The set of all idempotents of a semigroup is denoted $E(S)$. A semigroup $S$ will be called \textit{idempotent-free} if $E(S) = \emptyset$. An element $s$ of a semigroup $S$ is called \textit{indecomposable} if there are no $s_{1},s_{2} \in S$ such that $s = s_{1}s_{2}$ (and is otherwise called \textit{decomposable}). Similarly, an element $m$ of a monoid $M$ is called \textit{indecomposable} if there are no \linebreak $m_{1},m_{2} \in M\setminus \{1_{M}\}$ such that $m = m_{1}m_{2}$ (and is otherwise called \textit{decomposable}).

For a semigroup $S$, let $1$ be a symbol not in $S$, and define $S^{1} := S$ if $S$ has an identity, and $S\cup\{1\}$ otherwise, where $1$ acts as an identity on $S$. \textit{Green's relations} $\R, \L, \H, \J$ are the equivalence relations on $S$ that can be given by the following:
\begin{alignat*}{3}
& (s,t) \in \R &&\Leftrightarrow &&\, (\exists x,y \in S^{1})(s = tx)(t = sy); \\
& (s,t) \in \L &&\Leftrightarrow &&\, (\exists x,y \in S^{1})(s = xt)(t = ys); \\
& (s,t) \in \H &&\Leftrightarrow &&\, (s,t) \in \R\cap \L;\\
& (s,t)\in \J &&\Leftrightarrow &&\, (\exists x,x',y,y' \in S^{1})(s = xty)(t = x'sy').
\end{alignat*}

For $\K \in \{\R,\L,\H,\J\}$, we say that a semigroup $S$ is $\K$\textit{-trivial} if $\K = \{(s,s) : s \in S\}$. Note that as $\H \subseteq \R \subseteq \J$ and $\H \subseteq \L \subseteq \J$, in particular if a semigroup is $\J$-trivial, it is also $\R$-trivial, $\L$-trivial and $\H$-trivial.

A semigroup $S$ is said to be \textit{finitely generated} if there exists a finite subset $X$ of $S$ such that $S = \langle X \rangle$ (i.e the elements of $S$ are expressible as finite products of elements in $X$). Given a semigroup $S$ and $X\subseteq $ a finite generating set, the \textit{word problem of } $S$ with respect to $X$ is given by \[\text{WP}(S,X) = \{(u,v) \in X^{+} \times X^{+} : u =_{S} v\},\] where $u =_{S} v$ if $u$ and $v$ represent the same word in $S$ when written as words over the generating set $X$. The word problem of $S$ is said to be \textit{decidable} with respect to $X$ if there exists an algorithm taking $S$, a finite generating set $X \subseteq S$ and any $(u,v)\in X^{+} \times X^{+}$ as inputs which determines whether or not $(u,v) \in \text{WP}(S,X)$. 

Given a finitely generated semigroup $S$, a finitely generated subsemigroup $T$ of $S$ and a generating set $X$ for $S$, the \textit{generalized word problem of} $T$ \textit{ in } $S$ is the set of words over $X$ which represent an element in $T$. The generalized word problem is said to be decidable if there is an algorithm taking $S, X$ and a finite subset $Y$ of $X^{*}$ generating $T$, which decides whether or not a word $w$ over $X$ represents an element in $\langle Y\rangle$. 

\section{Fiber products of free semigroups/monoids over finite fiber quotients}
\label{sec2}
\setcounter{propn}{0}
This section is devoted to classifying the finite fiber quotients $F$ and associated epimorphisms $\varphi, \psi$ with free semigroup/monoid domains for which $\fibpro$ is finitely generated.

We begin by showing in the free semigroup case, there are no such fiber quotients.
\begin{propn}\thlabel{FreeSemiFin} Let $\varphi: A^{+}\rightarrow S$, $\psi: B^{+} \rightarrow S$ be epimorphisms where $S$ is a finite semigroup. Then the fiber product $\fibpro$ of $A^{+}$ with $B^{+}$ over $S$ with respect to $\varphi, \psi$ is not finitely generated.
\begin{proof}
Let $(u,v) \in \fibpro$. Then there exists some $s \in S$ such that $\varphi(u) = \psi(v) = s$. As $S$ is finite, there exists some $k \in \mathbb{N}$ such that $s^{k}$ is idempotent. Hence $(u^{k},v^{nk}) \in \fibpro$ for all $n \in \mathbb{N}$.

 Suppose for a contradiction that $X = \{(u_{i},v_{i}) : 1 \leq i \leq p\}  \subseteq A^{+} \times B^{+}$ were a finite generating set for $\fibpro$. Then as $u^{k}$ can be decomposed into at most $k|u|$ factors in $A^{+}$, it follows that each pair $(u^{k},v^{nk})$ can be decomposed into at most $k|u|$ factors in $\langle X \rangle$. This is a contradiction, as this implies that $|v^{nk}| \leq k|u|\max_{1 \leq i \leq p}|v_{i}|$ for all $n \in \mathbb{N}$. Hence $\fibpro$ is not finitely generated.
\end{proof}
\end{propn}

%
%

For the remainder of this section, we work towards giving necessary and sufficient conditions for fiber products of two finitely generated free monoids over finite fiber quotients to be finitely generated. Our next lemma shows that such quotients are necessarily restricted to the class of finite groups.

\begin{lemma}Let $\varphi : A^{*} \rightarrow M$, $\psi : B^{*} \rightarrow M$ be epimorphisms onto a finite monoid $M$. If $M$ is not a group, then $\fibpro$ is not finitely generated. \thlabel{FreeMonGroupFibers}
\begin{proof}
$\psi(B)$ is a generating set for $M$ by surjectivity. As $M$ is finite monoid which is not a group, then there exists some $m \in \psi(B)$  and $k \in \mathbb{N}$ such that $m^{k}$ is idempotent, but $m^{k} \not = 1_{M}$. 

As $\varphi$ and $\psi$ are surjections, then there exists a word $u \in A^{*}$ and a letter $b \in B$ such that $\varphi(u) = m = \psi(b)$. Hence $\{(u^{k},b^{nk}) : n \in \mathbb{N}\} \subseteq \fibpro.$

Suppose for a contradiction that $X = \{(u_{i},v_{i}) : 1 \leq i \leq p\} \subseteq A^{*} \times B^{*} $ were a finite generating set for $\fibpro$. As $\psi(b^{j}) \not = 1_{M}$ for all $j \in \mathbb{N}$ then it follows that \linebreak $(\epsa,b^{j}) \not \in \fibpro$ for all $j \in \mathbb{N}$, hence we must have $(u^{k},b^{nk}) \in \langle X' \rangle$ for all $n \in \mathbb{N}$, where $X' = \{(u_{i},v_{i}) \in X : u_{i} \not = \varepsilon\}$. Then as $u^{k}$ can be decomposed into at most $k|u|$ non-empty factors in $X^{*}$, it follows that each pair $(u^{k},b^{nk})$ can be decomposed into at most $k|u|$ factors in $\langle X' \rangle$. This is a contradiction, as this implies that $|b^{nk}| \leq k|u|\max_{1 \leq i \leq p}|v_{i}|$ for all $n \in \mathbb{N}$. Hence $\fibpro$ is not finitely generated.
\end{proof}
\end{lemma}

Our next lemma refines the previous result, to show that the fiber quotients of interest must be cyclic groups.

\begin{lemma}Let $\varphi : A^{*} \rightarrow G$, $\psi : B^{*} \rightarrow G$ be epimorphisms where $G$ is a finite non-cyclic group. Then $\fibpro$ is not finitely generated. \thlabel{FreeMonCyclicFibers}
\begin{proof} $\psi(B)$ is a finite generating set for the group $G$ by surjectivity. As $G$ is not cyclic, then there exist elements $g,h \in \psi(B)$ such that 
\begin{equation} g \not = h^{p} \text{ and } h \not = g^{p} \text{ for all } p \in \mathbb{N}_{0}. \label{notcyc}\end{equation} Note that \eqref{notcyc} implies $gh^{p} \not = 1_{G}$ for any $p \in \mathbb{N}$, By surjectivity, there exist distinct letters $a,b \in B$ such that $\psi(a) = g$ and $\psi(b) = h$. In particular, $\psi(ab^{p}) \not = 1_{G}$ for any $p \in \N$.

As $G$ is a finite group, let $j,k$ denote the orders of the elements $g$, $h$ respectively. Note in particular that $j,k > 1$. Then it follows that $\psi(ab^{nk}a^{j-1}) = 1_{G}$ for all $n \in \mathbb{N}$. Hence $$\{ (\epsa,ab^{nk}a^{j-1}) : n \in \mathbb{N}\} \subseteq \fibpro$$.

We claim that $(\epsa, ab^{nk}a^{j-1})$ for $n \in \mathbb{N}$ is indecomposable in $\fibpro$. Suppose to the contrary that there were a non-trivial decomposition
\begin{equation*}(\epsa,ab^{nk}a^{j-1}) = \prod_{i=1}^{p} (u_{i},v_{i}). \end{equation*}
Clearly, it must be the case that each $u_{i} = \varepsilon$. Hence none of the $v_{i}$ are empty. Consider the subword $v_{1}$. Then $a$ must be a prefix of $v_{1}$, as $v_{1}$ is not empty. As $\psi(a) \not = 1_{G}$, then $v_{1} \not = a$, and hence $ab$ is a prefix of $v_{1}$. As $\psi(ab^{p}) \not = 1_{G}$ for any $p \in \mathbb{N}$, it follows that $ab^{nk}a$ is a prefix of $v_{1}$. Finally, as $\psi(a^{p}) \not = 1_{G}$ for any $1 \leq p < j$, it must be that $v_{1} = ab^{nk}a^{j-1}$. Hence the claim is proved, and as any generating set for $\fibpro$ must contain the indecomposable elements of $\fibpro$, then $\fibpro$ is not finitely generated. 
\end{proof}
\end{lemma}

Finally, we give all conditions on finite fiber quotients and epimorphisms for which the associated fiber product of two free monoids is finitely generated.

\begin{thm}
Let $\varphi : A^{*} \rightarrow F$, $\psi : B^{*} \rightarrow F$ be epimorphisms where $F$ is finite. Then the fiber product of $A^{*}$ with $B^{*}$ over $F$ with respect to $\varphi, \psi$ is finitely generated if and only if $|\varphi(A)| = |\psi(B)| = 1$, and $F$ is a cyclic group.

\thlabel{FreeMonIff}
\begin{proof}

If $F$ is not a cyclic group, then $\fibpro$ is not finitely generated by \thref{FreeMonGroupFibers} and \thref{FreeMonCyclicFibers}. Otherwise, let $F = \text{Gp}\langle x : x^{n} = 1\rangle$, and suppose $\varphi : A^{*} \rightarrow F$ is such that $|\varphi(A)| > 1$. Then $\varphi(a) \not = \varphi(a')$ for some $a,a' \in A$, and we can choose $a_{1} \in \{a, a'\}$ such that $\varphi(a_{1})\not = 1$. We can also choose $a_{2} \in \{a,a'\}$ such that $\varphi(a_{1}a_{2}) \not = 1$, for otherwise $\varphi(a_{1}a) = \varphi(a_{1}a') \implies \varphi(a)= \varphi(a')$. 

Repeating this process, we can construct an arbitrarily long word $a_{1}a_{2}...a_{n} \in A^{+}$ such that $\varphi(a_{1}a_{2}\hdots a_{i}) \not = 1$ for $1 \leq i \leq n$. Letting $g = \varphi(a_{1}a_{2} \hdots a_{n})$, there exists some $u \in A^{+}$ of minimal length such that $g^{-1} = \varphi(u)$. Hence we can choose a sequence of words $\{w_{n}\}_{n \in \mathbb{N}} \subseteq A^{+}$ such that $|w_{n}| < |w_{n+1}|$, $\varphi(w_{n}) = 1$, but $\varphi(v) \not = 1$ for any prefix $v$ of $w_{n}$. It then follows that
$$\{(\epsa,w_{n}) : n \in \mathbb{N}\} \subseteq \fibpro,$$
\noindent and each pair $(\epsa,w_{n})$ is indecomposable in $\fibpro$, for otherwise there would exist a proper prefix $v$ of $w_{n}$ such that $\varphi(v) = 1$.

If $|\psi(B)| > 1$ is such that $\psi(b)\not = \psi(b')$ for some $b,b' \in B$, then the same argument also shows that $\fibpro$ is not finitely generated.

$\left(\Leftarrow\right)$ Supposing that $F$ is a cyclic group and $\varphi, \psi$ satisfy the conditions of the theorem, then $F = \text{Gp}\langle x : x^{n} = 1 \rangle$ for some $n \in \mathbb{N}$, and $\varphi(A) = \{x^{p}\}$, $\psi(B) = \{x^{q}\}$ for some $1 \leq p,q \leq n$ with $\text{gcd}(p,n) = \text{gcd}(q,n) = 1$. If $(u,v) \in \fibpro$, then
\begin{eqnarray*}
&& \varphi(u) = \psi(v) \\
\Leftrightarrow && (x^{p})^{|u|} = (x^{q})^{|v|} \\
\Leftrightarrow && x^{p|u| \,(\text{mod}\,n)} = x^{q|v|\,(\text{mod}\,n)} \\
\Leftrightarrow && p|u| \,(\text{mod}\,n) = q|v|\,(\text{mod}\,n).
\end{eqnarray*}
Hence 
$$\fibpro = \{(u,v) \in A^{*} \times B^{*} : p|u| \equiv q|v| \,(\text{mod}\,n)\}.$$
We claim that $\fibpro$ is finitely generated by
\begin{alignat*}{3}X  =& &&  \{(u,v) \in A^{*} \times B^{*}: p|u| \equiv q|v| \,(\text{mod}\,n), 0 \leq |u|,|v| \leq n\} \\ 
 &\setminus&& \{(u,v) \in A^{*}\times B^{*} : |u| = |v| = n\}.\end{alignat*}
Let $(u,v) \in \fibpro$. Let $p'$ be such that $pp' \equiv 1 \, (\text{mod}\,n)$, and let $q'$ be such that $qq' \equiv 1 \, (\text{mod}\,n)$. Then as $|u| = k_{1}n + r_{1}$ for some $0 \leq r_{1} < n$ with $r_{1} \equiv p'q|v| \,(\text{mod}\,n)$, $k_{1} \in \mathbb{N}_{0}$, and $|v| = k_{2}n + r_{2}$ for some $0 \leq r_{2} < n$ with $r_{2} \equiv q'p|u| \,(\text{mod}\,n)$, $k_{2} \in \mathbb{N}_{0}$, it follows that
$$u = u'x_{1}x_{2}\hdots x_{k_{1}}$$ 
for some $u', x_{i} \in A^{*}$ with $|u'| = r_{1}$, $|x_{i}| = n$ for $1 \leq i \leq k_{1}$, and similarly
$$v = v'y_{1}y_{2}\hdots y_{k_{2}}$$
for some $v', y_{i} \in B^{*}$ with $|v'| = r_{2}$, $|y_{i}| = n$ for $1 \leq i \leq k_{2}$. Then we have:
\begin{itemize}
\item $(u',v') \in X$, as $0 \leq |u'|,|v'| \leq n$, and $$p|u'| = pr_{1} \equiv pp'q|v| \equiv q|v| \equiv qr_{2} \equiv q|v'| \, (\text{mod}\,n);$$
\item $(x_{i},\epsb) \in X$ for all $1 \leq i \leq k_{1}$, as $p|x_{i}| \equiv 0 \equiv q|\epsb| \,(\text{mod}\,n)$;
\item $(\epsa,y_{i}) \in X$ for all $1 \leq i \leq k_{2}$, as $q|y_{i}| \equiv 0 \equiv p|\epsa| \,(\text{mod}\,n)$.
\end{itemize}

Hence $$(u,v) = (u',v')(x_{1},\epsb)\hdots(x_{k_{1}},\epsb)(\epsa, y_{1})\hdots(\epsa,y_{k_{2}}) \in \langle X \rangle,$$
proving the claim.\end{proof}	
\end{thm}

\begin{thm}\thlabel{fingenifffinpres}
Let $\varphi : A^{*} \rightarrow F$, $\psi : B^{*} \rightarrow F$ be epimorphisms where $F$ is finite. If the fiber product of $A^{*}$ with $B^{*}$ over $S$ with respect to $\varphi, \psi$ is finitely generated, then it is is also finitely presented.
\end{thm}

We introduce the next two Lemmas in order to prove \thref{fingenifffinpres}:
\begin{lemma}\thlabel{RelLem1}
Let $\varphi : A^{*} \rightarrow F$, $\psi : B^{*} \rightarrow F$ be epimorphisms (where \linebreak $F = \textup{Gp}\langle x : x^{n} = 1 \rangle$) satisfying $\varphi(A) = \{x^{p}\}, \psi(B) = \{x^{q}\}$ for some $1 \leq p,q \leq n$ with $\textup{gcd}(p,n) = \textup{gcd}(q,n) = 1$. Let \[\bar{\Gamma} := \{ \gamma{(u,v)} : u \in A^{*}, v \in B^{*}, \,p|u| \equiv q|v|\,\textup{mod}\,n, \,0 \leq |u|,|v| \leq n\}\]
and
\[\Gamma := \bar{\Gamma}\setminus \{\gamma(u,v) : u \in A^{*}, v \in A^{*}, |u| = |v| = n \textup{ or } |u| = |v| = 0\}\] be sets of formal symbols. Then the relations 
\begin{align}\label{r1}
&(\gamma(\epsa,v) \gamma(u,\epsb), \, \gamma(u,\epsb)\gamma(\epsa,v)) &&(|u| = |v| = n);\tag{R1} \\
&(\gamma(\epsa,v) \gamma(u,v_{1}), \, \gamma(u,v_{2})\gamma(\epsa,v_{3})) && ( 0 < |v_{1}| < n, |v| = |v_{3}| = n,\tag{R2} \label{r2}\\
& && |v_{1}| = |v_{2}|, vv_{1} = v_{2}v_{3});\nonumber\\
&(\gamma(u,\epsb)\gamma(u_{1},v) , \,\gamma(u_{2},v)\gamma(u_{3},\epsb)) &&(0 < |u_{1}| < n, |u| = |u_{3}| = n,\tag{R3}\label{r3}\\
& && |u_{1}| = |u_{2}| , uu_{1} = u_{2}u_{3})\nonumber
%
%
 \end{align}
over $\Gamma$ hold in $\fibpro$.
\begin{proof}
Let $\bar{\pi} : \Gamma \rightarrow \fibpro$ be given by $\bar{\pi}(\gamma(u,v)) = (u,v)$, and let $\pi : \Gamma^{*} \rightarrow \fibpro$ be the unique homomorphism extending $\bar{\pi}$. Then in the case of \eqref{r1}, we have
\[\pi(\gamma(\epsa,v)\gamma(u,\epsb)) =  (\epsa,v)(u,\epsb) = (u,v) = (u,\epsb)(\epsa,v) = \pi(\gamma(u,\epsb)\gamma(\epsa,v)),\] and so \eqref{r1} holds in $\fibpro$. In the case of \eqref{r2}, we have
\begin{align*}\pi(\gamma(\epsa,v)\gamma(u,v_{1})) &= (\epsa,v)(u,v_{1}) \\
&= (u,vv_{1}) \\
&= (u,v_{2}v_{3}) \\
&= (u,v_{2})(\epsa,v_{3})\\ 
&= \pi(\gamma(u,v_{2})\gamma(\epsa,v_{3})),\end{align*}
and hence \eqref{r2} holds in $\fibpro$. In the case of \eqref{r3}, we have
\begin{align*}\pi(\gamma(u,\epsb)\gamma(u_{1},v)) &= (u,\epsb)(u_{1},v) \\
&= (uu_{1},v) \\
&= (u_{2}u_{3},v)\\
&= (u_{2},v)(u_{3},\epsb)\\ 
&= \pi(\gamma(u_{2},v)\gamma(u_{3},\epsb)),\end{align*}
and hence \eqref{r3} holds in $\fibpro$.
\end{proof}
\end{lemma}

\begin{lemma}\thlabel{RelLem2}
Let $\fibpro$ and $\Gamma$ be as in \thref{RelLem1}. Let $\gamma(u_{1},v_{1}),\gamma(u_{2},v_{2}) \in \Gamma$ with \linebreak $0<|u_{1}|,|u_{2}| < n$, and define $u_{3},u_{4} \in A^{*}$, $v_{3},v_{4} \in B^{*}$ as follows:
\[\begin{aligned}
&u_{1}u_{2} = u_{3}u_{4}, && |u_{4}|= n && \textup{ if } |u_{1}u_{2}| > n, \\
&v_{1}v_{2} = v_{3}v_{4}, && |v_{4}|= n && \textup{ if } |v_{1}v_{2}| > n.
\end{aligned}\]

Then the relation \[(\gamma(u_{1},v_{1})\gamma(u_{2},v_{2}),\,w) \tag{R4}\label{r4}\]
where
\[w = \begin{cases} \gamma(u_{1}u_{2},v_{1}v_{2}) & \textup{if}\ |u_{1}u_{2}|, |v_{1}v_{2}| < n,\\
\gamma(u_{3},v_{1}v_{2})\gamma(u_{4},\epsb) & \textup{if}\ |u_{1}u_{2}| > n, |v_{1}v_{2}| < n,\\
\gamma(u_{1}u_{2},v_{3})\gamma(\epsa,v_{4}) & \textup{if}\ |u_{1}u_{2}| < n, |v_{1}v_{2}| > n,\\
\gamma(u_{3},v_{3})\gamma(u_{4},\epsb)\gamma(\epsa,v_{4}) & \textup{if}\ |u_{1}u_{2}|, |v_{1}v_{2}| > n, \\
\gamma(u_{1}u_{2},\epsb)\gamma(\epsa,v_{1}v_{2}) & \textup{if}\ |u_{1}u_{2}| =  |v_{1}v_{2}| = n \\
\end{cases}\]
\noindent over $\Gamma$ holds in $\fibpro$.
\begin{proof}
Let $\bar{\pi} : \Gamma \rightarrow \fibpro$ be given by $\bar{\pi}(\gamma(u,v)) = (u,v)$, and let \linebreak $\pi : \Gamma^{*} \rightarrow \fibpro$ be the unique homomorphism extending $\bar{\pi}$. For $\gamma(u_{1},v_{1}),\gamma(u_{2},v_{2}) \in \Gamma$ with $0 < |u_{1}|,|u_{2}| < n$, if $|u_{1}u_{1}|, |v_{1}v_{2}| < n$, then
\[\pi(\gamma(u_{1},v_{1})\gamma(u_{2},v_{2})) = (u_{1},v_{1})(u_{2},v_{2}) = (u_{1}u_{2},v_{1}v_{2}) = \pi(\gamma(u_{1}u_{2},v_{1}v_{2})).\]
If $|u_{1}u_{2}| > n$, $|v_{1}v_{2}|<n$, then
\begin{align*}\pi(\gamma(u_{1},v_{1})\gamma(u_{2},v_{2})) &= (u_{1},v_{1})(u_{2},v_{2}) \\
&= (u_{1}u_{2},v_{1}v_{2}) \\
&= (u_{3}u_{4}, v_{1}v_{2}) \\
&= (u_{3},v_{1}v_{2})(u_{4},\epsb) \\
&= \pi(\gamma(u_{3},v_{1}v_{2})\gamma(u_{4},\epsb)).\end{align*}

If $|u_{1}u_{2}|<n, |v_{1}v_{2}|>n$, then 
\begin{align*}\pi(\gamma(u_{1},v_{1})\gamma(u_{2},v_{2})) &= (u_{1},v_{1})(u_{2},v_{2}) \\
&= (u_{1}u_{2},v_{1}v_{2}) \\
&= (u_{1}u_{2},v_{3}v_{4}) \\
&= (u_{1}u_{2},v_{3})(\epsa,v_{4}) \\
&= \pi(\gamma(u_{1}u_{2},v_{3})\gamma(\epsa,v_{4})).\end{align*}

If $|u_{1}u_{2}|, |v_{1}v_{2}| > n$, then
\begin{align*}\pi(\gamma(u_{1},v_{1})\gamma(u_{2},v_{2})) & = (u_{1},v_{1})(u_{2},v_{2}) \\
& = (u_{1}u_{2},v_{1}v_{2}) \\
& = (u_{3}u_{4},v_{3}v_{4}) \\
& = (u_{3},v_{3})(u_{4},\epsb)(\epsa, v_{4}) \\
& = \pi(\gamma(u_{3},v_{3})\gamma(u_{4},\epsb)\gamma(\epsa, v_{4})).\end{align*}

Finally if $|u_{1}u_{2}| = |v_{1}v_{2}| = n$, then
\begin{align*}\pi(\gamma(u_{1},v_{1})\gamma(u_{2},v_{2})) & = (u_{1},v_{1})(u_{2},v_{2}) \\
& = (u_{1}u_{2},v_{1}v_{2}) \\
& = (u_{1}u_{2},\epsb)(\epsa, v_{1}v_{2}) \\
& = \pi(\gamma(u_{1}u_{2},\epsb)\gamma(\epsa, v_{1}v_{2})).\end{align*}

Noting that $|u_{1}u_{2}| = n \Leftrightarrow |v_{1}v_{2}| = n$, then we have all possible cases. Hence \eqref{r4} holds in $\fibpro$ as claimed.

\end{proof}

 \end{lemma}

We now use \thref{RelLem1} and \thref{RelLem2} alongside \thref{FreeMonIff} to prove \thref{fingenifffinpres}.

\begin{proof}[Proof of \thref{fingenifffinpres}]
By \thref{FreeMonIff}, if $\fibpro$ is finitely generated then \linebreak $|A|,|B| < \infty$, $F = \text{Gp}\langle x : x^{n} = 1 \rangle$ for some $n \in \N$, and $\varphi(A) = \{x^{p}\}$, $\psi(B) = \{x^{q}\}$ for some $1 \leq p,q \leq n$ with $\textup{gcd}(p,n) = \textup{gcd}(q,n) = 1$. Let \[\bar{\Gamma} := \{ \gamma{(u,v)} : u \in A^{*}, v \in B^{*}, \,p|u| \equiv q|v|\,\textup{mod}\,n, \,0 \leq |u|,|v| \leq n\}\]
and
\[\Gamma := \bar{\Gamma}\setminus \{\gamma(u,v) : u \in A^{*}, v \in A^{*}, |u| = |v| = n \textup{ or } |u| = |v| = 0\}\] be sets of formal symbols. Let $R$ be the set of relations on $\Gamma$ given by \eqref{r1}-\eqref{r4} in \thref{RelLem1} and \thref{RelLem2}. Let $\bar{\pi} : \Gamma \rightarrow \fibpro$ be given by $\bar{\pi}(\gamma(u,v)) = (u,v)$, and let $\pi : \Gamma^{*} \rightarrow \fibpro$ be the unique homomorphism extending $\bar{\pi}$. We will show that $\fibpro \cong \textup{Mon}\langle\Gamma : R\rangle$, by showing $\textup{ker}\,\pi = R^{\sharp}$.

By \thref{RelLem1} and \thref{RelLem2}, it follows that $\textup{ker}\,\pi \supseteq R$, and hence $\textup{ker}\,\pi \supseteq R^{\sharp}$.

To show that $\textup{ker}\,\pi \subseteq R^{\sharp}$, we make the following claims:

\textit{Claim 1:} For all $w \in \Gamma^{*}$, \begin{align}(w, w_{1}w_{2}w_{3}) \in R^{\sharp} \text{ for some } & w_{1} \in \{\gamma(u,v) \in \Gamma : 0< |u|,|v| < n\}\cup\{\varepsilon_{\Gamma^{*}}\},\nonumber\\
& w_{2} \in \{\gamma(u,v) \in \Gamma  : |u| = n\}^{*},\label{threewords}\\
& w_{3} \in  \{\gamma(u,v) \in \Gamma : |v| = n\}^{*}\nonumber.\end{align} We briefly adopt some terminology for letters in $\Gamma$ to this end. We will say that \linebreak $\gamma(u,v) \in \Gamma$ is of $\varepsilon$-type if either $u = \epsa$ or $v = \epsb$. Otherwise, $\gamma(u,v)$ will be of $\phi$-type. The following rewriting procedure proves the claim:

(A1) If a letter of $\phi$-type in $w$ is preceeded by a letter of $\varepsilon$-type, then \eqref{r2} and \eqref{r3} allows us to replace them with a letter of $\phi$-type proceeded by a letter of $\varepsilon$-type. Hence using a sequence of \eqref{r2} and \eqref{r3} allows us to rewrite $w$ as  $w'w''$, where $w'$ is a (possibly empty) word consisting of $\phi$-type letters, and $w''$ is a (possibly empty) word of $\varepsilon$-type. That is, $(w,w'w'') \in R^{\sharp}$ for some $w' \in \{\gamma(u,v) \in \Gamma : 0 <|u|,|v|<n\}^{*}$,\linebreak $w'' \in \{\gamma(u,v) \in \Gamma : u = \epsa \text{ or } v = \epsb\}^{*}$.\\
(A2)  As \eqref{r4} allows us to replace two concurrent $\phi$-type letters with a single $\phi$-type letter followed by at most two $\varepsilon$-type letters, then repeatedly using \eqref{r4} from right to left on the letters in $w'$ allows us to rewrite $w'w''$ as $w_{1}w'''$ for $w_{1}$ a $\phi$-type letter (or $\varepsilon_{\Gamma^{*}}$), and $w'''$ a (potentially empty) word consisting of $\varepsilon$-type letters. That is, $(w'w'',w_{1}w''') \in R^{\sharp}$ for some $w_{1}' \in \{\gamma(u,v) \in \Gamma : 0 <|u|,|v|<n\}\cup\{\varepsilon_{\Gamma^{*}}\}$, $w''' \in \{\gamma(u,v) \in \Gamma : u = \epsa \text{ or } v = \epsb\}^{*}$.\\
(A3) As \eqref{r1} allows us to swap the order of any two concurrent $\varepsilon$-type letters, then repeatedly using \eqref{r1} on the letters of $w'''$ allows us to rewrite $w'''$ as $w_{2}w_{3}$, where $w_{2} \in \{\gamma(u,v) \in \Gamma : |u| = n\}$, $w_{3} \in \{\gamma(u,v) \in \Gamma : |v| = n\}$. That is $(w_{1}w''',w_{1}w_{2}w_{3}) \in R^{\sharp}$ where $w_{1},w_{2},w_{3}$ are as in \eqref{threewords}.

 Hence $(w,w_{1}w_{2}w_{3}) \in R^{\sharp}$ as claimed.

\textit{Claim 2:} If $(w,w') \in \textup{ker}\,\pi$, then
 \begin{align*}(w,w_{1}w_{2}w_{3}),(w',w'_{1}w'_{2}w'_{3}) \in R^{\sharp} \text{ for some } & w_{1},w'_{1} \in \{\gamma(u,v) \in \Gamma : 0< |u|,|v| < n\}\cup\{\varepsilon_{\Gamma^{*}}\},\\
& w_{2},w'_{2} \in \{\gamma(u,v) \in \Gamma  : |u| = n\}^{*},\\
& w_{3},w'_{3} \in  \{\gamma(u,v) \in \Gamma : |v| = n\}^{*}\end{align*}
with $w_{1} = w'_{1}, w_{2} = w'_{2}, w_{3} = w'_{3}$.

The fact that $(w,w_{1} w_{2} w_{3}), (w',w'_{1} w'_{2} w'_{3}) \in R^{\sharp}$ for such $w_{1},w_{2},w_{3},w'_{1},w'_{2},w'_{3}$ follows from Claim 1. As $R^{\sharp} \subseteq \textup{ker}\,\pi$, then $(w,w_{1} w_{2}w_{3}),(w', w'_{1} w'_{2} w'_{3}) \in \text{ker}\,\pi$ also.

Firstly, consider for a contradiction that $w_{1} \not = w'_{1}$. Then as \linebreak $w_{1}, w'_{1} \in \{\gamma(u,v) \in \Gamma : 0< |u|,|v| < n\}\cup\{\varepsilon_{\Gamma^{*}}\}$, then $w_{1} = \gamma(u,v), w'_{1} = \gamma(u',v')$, where either $u \not = u'$ or $v \not = v'$. 

If $u \not = u'$, then if $|u| = |u'|$, it follows that $\pi(w) \not = \pi(w')$, as $\pi(w_{1}) = (u,v)$ and $\pi(w'_{1}) = (u',v')$ are non-equal prefixes of $\pi(w)$ and $\pi(w')$ respectively, contradicting $(w,w') \in \textup{ker}\,\pi$.

 Otherwise, if $|u| \not = |u'|$, then $\pi(w) \not = \pi(w')$, as the first coordinates in $\pi(w)$ and $\pi(w')$ are free words over $A$, whose lengths are congruent to $|u|$ and $|u'|$ modulo n, respectively.

The argument for the case where $v \not = v'$ is the same as for $u \not = u'$. Hence $\pi(w) \not = \pi(w')$ if $w_{1} \not = w'_{1}$, and so it must be that $w_{1}=w'_{1}$ to avoid contradiction.

Secondly, consider for a contradiction that $w_{2} \not = w'_{2}$. Then as \linebreak $w_{2},w'_{2} \in \{\gamma(u,v) \in \Gamma : |u| \in n\}^{*}$, it follows that $\pi(w_{2}) = (u,\epsb)$, $\pi(w'_{2}) = (u',\epsb)$ for some $u,u' \in A^{*}$ with $u \not = u'$. As $u$ and $u'$ are suffixes of the first coordinates (being free words over $A$) of $\pi(w)$ and $\pi(w')$ respectively, then it follows that $\pi(w) \not = \pi(w')$, as their first coordinates either have different suffixes in the case where $|u| = |u'|$, or are different in length in the case where $|u| \not = |u'|$. Hence it must be that $w_{2} = w'_{2}$ to avoid contradiction.

Finally, consider for a contradiction that $w_{3} \not = w'_{3}$. This is dual to the $w_{2} \not = w'_{2}$ consideration, as applying the same argument to the second coordinates of $\pi(w)$ and $\pi(w')$ will contradict $(w,w') \in \textup{ker}\,\pi$. Hence it must be that $w_{3} = w'_{3}$. This ends the proof of Claim 2.

By transitivity, it follows from Claim 2 that if $(w,w') \in \textup{ker}\,\pi$, then $(w,w') \in R^{\sharp}$ also, and hence $\textup{ker}\,\pi \subseteq R^{\sharp}$. As we have now shown $\textup{ker}\,\pi = R^{\sharp}$, then \[\fibpro \cong \Gamma^{*}/\text{ker}\,\pi = \Gamma^{*}/R^{\sharp} = \text{Mon}\langle \Gamma : R \rangle,\] and thus $\fibpro$ is finitely presented as required.
\end{proof}

\section{Infinite fiber quotients}\label{sec3}
From the results of the above section, one might expect that infinite fiber quotients would give rise to many more finitely generated fiber products of free semigroups and free monoids. We thus seek to classify some properties of fiber quotients which give finitely generated fiber products. In this section, we obtain some necessary semigroup theoretic conditions for finite generation in the general infinite fiber quotient case. We begin with observation that the quotient itself must at least be finitely generated.

\begin{lemma}\thlabel{fgcompsandfiberlemma}
Let $S,T,F$ be semigroups, and let $\varphi : S \rightarrow F$, $\psi : T \rightarrow F$ be epimorphisms. If $\fibpro$ is finitely generated, then $S$, $T$ and $F$ are finitely generated.
\begin{proof}
Suppose that $\fibpro$ is finitely generated, and let $$V := \{(s_{i},t_{i}) : 1 \leq i \leq n\} \subseteq \fibpro$$ be a generating set for $\fibpro$. Then $\pi_{1}(V) = \{s_{i} : 1 \leq i \leq n\}$ generates $S$ (as $\fibpro$ is subdirect), $\pi_{2}(V) = \{t_{i} : 1 \leq i \leq n\}$ generates $T$, and as $\varphi$ is a surjection, it follows that $\varphi(\pi_{1}(V))$ is a generating set for $F$.\end{proof}
\end{lemma}

We also note in the next two results that finite generation of a fiber product of two free semigroups/monoids is equivalent to the fiber product having finitely many indecomposable elements.

\begin{lemma}\thlabel{FGiffIndecompFin}
Let $F$ be a semigroup, and let $\varphi : A^{+} \rightarrow F$, $\psi : B^{+} \rightarrow F$ be two epimorphisms with $A,B$ finite alphabets. Then $\fibpro$ is finitely generated if and only if $\fibpro$ has finitely many indecomposable elements.

\begin{proof}
As every generating set for $\fibpro$ contains the set of indecomposable elements, sufficiency is immediate.

For necessity, we show that the set of indecomposable elements generates $\fibpro$. Let $(u,v) \in \fibpro$. If $(u,v)$ is indecomposable, then there is nothing to show. Otherwise, if $(u,v) \in \fibpro^{2}$, then there exist $(u',v'),(u'',v'') \in \fibpro$ with \linebreak $(u,v) = (u',v')(u'',v'')$. As $u \in A^{+}$, $v \in B^{+}$, it follows that $|u'|,|u''| < |u|$ and $|v'|, |v''| < |v|$. 

Repeating this factoring process on $(u',v')$ or $(u'',v'')$ if either are decomposable, and so on with their decomposable factors, then as the lengths of the words in the factors of $(u,v)$ decrease in every factorisation, it follows that this process is finite. Hence
\[ (u,v) = (u_{1},v_{1})(u_{2},v_{2})\dots (u_{n},v_{n})\]
where $(u_{i},v_{i})$ are indecomposable for $1 \leq i \leq n$, and $n \leq \text{max}\{|u|,|v|\}$. Thus $(u,v)$ is generated by indecomposable elements and thus $\fibpro$ is finitely generated, and so the result follows.
\end{proof}
\end{lemma}

\begin{lemma}\thlabel{FGMiffIndecompFin}
Let $F$ be a monoid, and let $\varphi : A^{*} \rightarrow F$, $\psi : B^{*} \rightarrow F$ be two epimorphisms with $A,B$ finite alphabets. Then $\fibpro$ is finitely generated if and only if $\fibpro$ has finitely many indecomposable elements.
\begin{proof}
The proof is the same as in \thref{FGiffIndecompFin}, noting that any decomposition of \linebreak $(u,v) = (u',v')(u'',v'')$ implies either $|u'|,|u''| < |u|$ or $|v'|,|v''| < |v|$.
%
%
\end{proof}
\end{lemma}

Our next result shows that no finitely generated fiber product of two free semigroups can have a fiber quotient containing a finite subsemigroup.

\begin{propn}\thlabel{torsionfree}
Let $\varphi : A^{+} \rightarrow S$, $\psi : B^{+} \rightarrow S$ be two epimorphisms onto a semigroup fiber $S$. If the fiber product of $A^{+}$ with $B^{+}$ over $S$ with respect to $\varphi, \psi$ is finitely generated, then $S$ is idempotent-free. 

\begin{proof}
Suppose to the contrary, that $e^{2} = e$ for some $e \in S$. Then by surjectivity, there exists $u \in A^{+}, v \in B^{+}$ such that $\varphi(u) = e = \psi(v)$. Then for all $n \in N$, as $\varphi(v^{n}) = e^{n} = e$, it follows that $(u,v^{n}) \in \fibpro$.

 If $X = \{(u_{i},v_{i}) : 1 \leq i \leq p\}  \subseteq A^{+} \times B^{+}$ were a finite generating set for $\fibpro$, then as $u$ can be decomposed into at most $|u|$ factors in $A^{+}$, it follows that each pair $(u,v^{n})$ can be decomposed into at most $|u|$ factors in $\langle X \rangle$. This is a contradiction, as this implies that $|v^{n}| \leq |u|\max_{1 \leq i \leq p}|v_{i}|$ for all $n \in \mathbb{N}$.
\end{proof}
\end{propn}

Our next result tells us about Green's relations on fiber quotients of fiber products of free semigroups, and in particular that they are all equal to the trivial relation.

\begin{propn}\thlabel{jtrivfiber}
Let $\varphi : A^{+} \rightarrow S$, $\psi : B^{+} \rightarrow S$ be two epimorphisms onto a semigroup fiber $S$. If the fiber product of $A^{+}$ with $B^{+}$ over $S$ with respect to $\varphi, \psi$ is finitely generated, then $S$ is $\J$-trivial.

\begin{proof}
Suppose to the contrary, that there exist some $s,t \in S$ with $(s,t) \in \J$, but $s \not = t$. Then in particular, there exists $x,y,x',y' \in S^{1}$ such that $s = xty$, $t = x'sy'$. In particular, $s = (xx')^{n}t(y'y)^{n}$ for all $n \in \N$. By surjectivity, there exists $u \in A^{+}$ and $v,w,w'\in B^{+}$ such that $\varphi(u) = s$, $\psi(v) = t$, $\psi(w) = xx'$, and $\psi(w') = y'y$. Hence $(u,w^{n}v(w')^{n} \in \fibpro$ for all $n \in \N$. If $X = \{(u_{i},v_{i}) : 1 \leq i \leq p\}  \subseteq A^{+} \times B^{+}$ were a finite generating set for $\fibpro$, then as $u$ can be decomposed into at most $|u|$ factors in $A^{+}$, it follows that each pair $(u,w^{n}v(w')^{n})$ can be decomposed into at most $|u|$ factors in $\langle X \rangle$. This is a contradiction, as this implies that $|w^{n}v(w')^{n}| \leq |u|\max_{1 \leq i \leq p}|v_{i}|$ for all $n \in \mathbb{N}$.
\end{proof}
\end{propn}

In the next result, we show that the properties of being $\J$-trivial and idempotent-free are not sufficient conditions for finite generation of fiber products of free semigroups. We draw an analogy with \thref{FreeMonIff} by choosing the free monogenic semigroup as a fiber, and show in particular that conditions on the associated homomorphisms are again necessary.

\begin{thm}\thlabel{Ntheorem} Let $A^{+},B^{+}$ be two free semigroups, and let $\varphi : A^{+} \rightarrow \N$, $\psi : B^{+} \rightarrow \N$ be two epimorphisms. Then the fiber product $\fibpro$ of $A^{+}$ with $B^{+}$ over $\N$ with respect to $\varphi, \psi$ is finitely generated if and only if either $|\varphi(A)| = 1$ or $|\psi(B)| = 1$.

\begin{proof}
$(\Rightarrow)$ We show the contrapositive. If $|\varphi(A)|, |\psi(B)| > 1$, then assume without loss that there exists $a \in A$, $b \in B$ such that $a\varphi = m$, $b\psi = n$, for some $m\geq n > 1$. As $\varphi$ is surjective, there exists some $x \in A$, $y \in B$ such that $\varphi(x) = 1 = \psi(y)$. \\


Note that $m = qn + r$ for some $q \in \N$, $0 \leq r < n$. As $\varphi(xa^{k}) = 1 +km = \varphi((b^{q}y^{r})^{k}y)$ (where $y^{0}$ is taken to be the empty word), it follows that $(xa^{k},(b^{q}y^{r})^{k}y) \in C$ for all $k \in \N$. We claim that $(xa^{k},(b^{q}y^{r})^{k}y)$ is irreducible in $C$, for all $k \in \N$. For otherwise, $(xa^{k},(b^{q}y^{r})^{k}y)= (u,v)(u',v')$ for some $u,u' \in A^{+}$, $v,v' \in B^{+}$, where $u$ and $v$ are proper prefixes of $xa^{k}$ and $(b^{q}y^{r})^{k}y$ respectively.\\

 Any proper prefix $u$ of $xa^{k}$ is such that $\varphi(u) \equiv 1\,\text{mod}\,m$, but any proper prefix $v$ of $(b^{q}y^{r})^{k}y$ is such that $v\psi \equiv j\,\text{mod}\,m$, where $j \in \{0, n,2n,\hdots,qn,qn+1,qn+2,\hdots,qn+r-1\}$. As $n \not = 1$ by assumption, and $k$ comes from a subset of least positive residues modulo $m$, it follows that $\psi(v) \not \equiv 1 \,\text{mod}\,m$, contradicting that $(u,v) \in C$. This proves the claim, and hence as any generating set for $C$ must contain $\{(xa^{k},(b^{q}y^{r})^{k}y) : k \in \N\}$, it follows that $C$ is not finitely generated.\\

$(\Leftarrow)$ It is enough to prove the statement assuming $|\varphi(A)| = 1$ without loss. As this is equivalent to $a\varphi = 1$ for all $a \in A$ by surjectivity of $\varphi$, it follows that
$$C = \{(u,v) \in A^{+}\times B^{+} : |u| = \psi(v)\}.$$

We claim that $C$ is generated by the set
$$X = \{(u,v) \in C : v \in B, |u| = \psi(v)\}$$
which is finite, as $B$ is finite and hence $|u|$ is bounded for $(u,v) \in X$.

Clearly $\langle X \rangle \subseteq C$. To prove the opposite containment, let $(u,v) \in C$. Then as \linebreak $u = a_{1}a_{2}\hdots a_{|u|}$ for some $a_{1},\hdots,a_{|u|} \in A$, and $v = b_{1}b_{2}\hdots b_{|v|}$ for some $b_{1},\hdots,b_{|v|} \in B$, it follows that $|u| = \psi(v) = \sum_{i=1}^{|v|}\psi(b_{i})$. As
\begin{equation}(u,v) = (a_{1}a_{2}\hdots a_{|u|}, b_{1}b_{2}\hdots b_{|v|}) = \prod_{i=0}^{|v|-1} (a_{j_{i} + 1}a_{j_{i} + 2}\hdots a_{j_{i} + (b_{i+1})\psi}, b_{i+1}) \in \langle X \rangle \label{decomp}\end{equation}
where $$j_{i} = \left\{\begin{matrix*}[l] 0 & \text{ if } i = 0 \\ \displaystyle \sum_{k=1}^{i}\psi(b_{k}) & \text{ otherwise,}\end{matrix*}\right. $$
the claim holds, as \eqref{decomp} gives a decomposition of $(u,v)$ into a product of elements in $X$.
\end{proof}
\end{thm}

Our next example notes that the above result does not, however, generalise to the class of free commutative semigroups.

\begin{exam}\thlabel{freecommsgp} Let $A^{+},B^{+}$ be two free semigroups, and let $\varphi : A^{+} \rightarrow F$, $\psi : B^{+} \rightarrow F$ be two epimorphisms onto a free commutative semigroup $F$ of (finite) rank larger than one. Then the fiber product $C$ of $A^{+}$ with $B^{+}$ over $F$ with respect to $\varphi, \psi$ is not finitely generated. 
\begin{proof}
Let $x,y \in F$ be two generators for $F$. Then as $\varphi, \psi$ are surjections, there exists $a,a' \in A$, $b,b' \in B$ such that $\varphi(a) = \psi(b) = x$, and $\varphi(a') =\psi( b') = y$. As $F$ is commutative, then $x^{n}y = yx^{n}$ for all $n \in \N$, and so it follows that $(a^{n}a', b'b^{n}) \in C$ for all $n \in \N$.

 As any proper prefix $u$ of $a^{n}a'$ is a power of $a$, it follows that $\varphi(u)$ is a power of $x$. But as any proper prefix $v$ of $b'b^{n}$ begins with $b'$, it follows that $\psi(v)$ contains a $y$. Hence $(a^{n}a',b'b^{n})$ is indecomposable in $C$ for all $n \in N$, as there are no proper prefixes $u$ of $a^{n}a'$, $v$ of $b'b^{n}$ such that $(u,v) \in C$.\end{proof}
\end{exam}

\section{Decision problems for free quotients}\label{decsec}

Perhaps the most natural example of semigroups satisfying the necessary conditions of finite generation given above in \thref{fgcompsandfiberlemma}, \thref{torsionfree}, and \thref{jtrivfiber} are the finitely generated free semigroups and monoids. Hence in this section, we consider some decision problems for fiber products of free semigroups/monoids with free fiber quotients. Our first observation establishes an equivalence between word problems in the fiber product and fiber quotient.

\begin{lemma}
Let $S,T,F$ be semigroups, and let $\varphi : S \rightarrow F$, $\psi : T \rightarrow F$ be epimorphisms, with $\fibpro$ finitely generated. Then the generalized  word problem for $\fibpro$ in $S\times T$ is decidable if and only if the word problem of $F$ is decidable.

\begin{proof}
Let $X \subseteq S \times T$ be any generating set for $S\times T$ and let $w$ be a word over $X$. Then $w \in \fibpro \Leftrightarrow \varphi(\pi_{S}(w)) = \psi(\pi_{T}(w))$. As $F$ is finitely generated by \thref{fgcompsandfiberlemma}, then by writing $\varphi(\pi_{S}(w))$ and $ \psi(\pi_{T}(w))$ as words over any finite generating set $Y$, it is decidable whether or not $(\varphi(\pi_{S}(w)), \psi(\pi_{T}(w)) \in \text{WP}(F,Y)$ and hence whether or not $w$ represents a word in $\fibpro$. Hence the generalized word problem for $\fibpro$ in $S \times T$ is decidable.\\

For the reverse direction, let $Y\subseteq F$ be any finite generating set for $F$ (whose existence is given by \thref{fgcompsandfiberlemma}) and $u,v$ be words over $Y$. As $\varphi, \psi$ are surjections, then there exists $(s,t) \in S \times T$ such that $(\varphi(s),\psi(t)) = (u,v)$.  Given a finite generating set $X$ for $S\times T$ and writing $(s,t)$ as a word over $X$, it is decidable whether or not $(s,t)$ represents a word in $\fibpro$ by decidability of the generalized word problem for $\fibpro$ in $S\times T$, and hence whether or not $(u,v) \in \text{WP}(F,Y)$. Hence the word problem of $F$ is decidable.
\end{proof}
\end{lemma}

For the remainder of this section, we seek to answer the following decision question:
\begin{que}\thlabel{decidabilityquestionforfib}
Is the finite generation problem for a fiber product of two free monoids with a free monoid fiber quotient decidable?
\end{que}

To answer this, we use a two tape automaton construction, for which we use the following definition:

\begin{defn}
A \textit{two-tape automaton} is a $6$-tuple $\mathcal{A} = (Q,\Sigma_{1},\Sigma_{2},\delta,\iota,F)$, where $Q$ is a finite set of states, $\Sigma_{1}$, $\Sigma_{2}$ are two input alphabets, $\delta \subseteq Q \times \Sigma \times Q$ is the transition relation (where $\Sigma = (\Sigma_{1}\cup\{\varepsilon_{1}\})\times (\Sigma_{2}\cup\{\varepsilon_{2}\})$, and $\varepsilon_{1},\varepsilon_{2}$ are the empty words over $\Sigma_{1},\Sigma_{2}$ respectively), $\iota \in Q$ is the initial state, $F \subseteq Q$ is the set of final states. 

An \textit{input} is a pair of words $(u,v) \in \Sigma_{1}^{*} \times \Sigma_{2}^{*}$. A two-tape automaton $\mathcal{A}$ \textit{accepts the input} $(u,v)$ if there exists a finite sequence of transitions $(q_{i-1},\sigma_{i},q_{i})_{i=1}^{k}$ where $q_{0} = \iota$, $q_{k} \in F$, and $(u,v) = \sigma_{i}\dots \sigma_{k}$. The \textit{language accepted by} $\mathcal{A}$ is the set $\mathcal{L}(\mathcal{A})$ of all inputs accepted by $\mathcal{A}$.

Finally, a \textit{cycle} of a two-tape automaton is a finite sequence of transitions $(q_{i-1},\sigma_{i},q_{i})_{i=1}^{k}$ where $q_{0} = q_{k}$. \\
\end{defn}

The construction process used in answering \thref{decidabilityquestionforfib} is then as follows. Let \linebreak $\varphi : A^{*} \rightarrow C^{*}$, $\psi : B^{*} \rightarrow C^{*}$ be two epimorphisms with $A,B,C$ finite alphabets. Let $\mathcal{A_{\varphi,\psi}} = (Q,\Sigma_{1},\Sigma_{2},\delta,\iota,F)$ be the associated two tape automaton (where an input $(u,v) \in A^{*} \times B^{*}$) given by the following:
\begin{itemize}
\item $Q :=  Q_{1} \cup Q_{2}\cup\{\iota\}\cup\{(\epsc,\epsc)\}$ where
\begin{align*}Q_{1} &:= \{(u,\epsc) \in C^{+}\times \{\epsc\} : (\exists w \in \varphi(A))(u <_{s} w)\};\\
Q_{2} & := \{(\epsc,v) \in \{\epsc\}\times C^{+} : (\exists w \in \psi(B))(v <_{s} w)\}; \end{align*} 
\item $\Sigma_{1} := A$, $\Sigma_{2} := B$.
\item $\delta = \bigcup_{i=1}^{8} \Delta_{i}\subset Q \times \Sigma \times Q$, where
\begin{align*}
\Delta_{1} & = \{(\iota,(a,\epsb),(\epsc,\epsc)) : a \in A, \varphi(a) = \epsc\} \\
\Delta_{2} & = \{(\iota,(\epsa,b),(\epsc,\epsc)) : b \in B, \psi(b) = \epsc\} \\
\Delta_{3} & =  \{(\iota,(a,b),(\psi(b)^{-1}\varphi(a),\epsc)) : a \in A, b \in B, \psi(b) \leq_{p} \varphi(a), \psi(b) \not = \epsc\}; \\
\Delta_{4} & =  \{(\iota,(a,b),(\epsc,\varphi(a)^{-1}\psi(b)) :  a \in A, b \in B, \varphi(a) \leq_{p} \psi(b), \varphi(a) \not = \epsc\}; \\
\Delta_{5} & =  \{((u,\epsc),(\epsa,b),(\psi(b)^{-1}u,\epsc))  :  b \in B, u \not = \epsc, \, \psi(b) \leq_{p} u\}; \\
\Delta_{6} & =  \{((u,\epsc),(\epsa,b),(\epsc,u^{-1}\psi(b))) :  b \in B, u \not = \epsc, \, u \leq_{p} \psi(b) \}; \\
\Delta_{7} & =  \{((\epsc,v),(a,\epsb),(\epsc,\varphi(a)^{-1}v) ) : a \in A, v \not = \epsc, \, \varphi(a) \leq_{p} v \}; \\
\Delta_{8} & =  \{((\epsc,v),(a,\epsb),(v^{-1}\varphi(a),\epsc ) : a \in A, v \not = \epsc, \, v \leq_{p} \varphi(a) \};
\end{align*}
\item $\iota$ is the initial state, $F = \{(\epsc,\epsc)\}$ is the set of final states.
\end{itemize}

\begin{exam}

Let $\varphi : \{a,b\}^{*} \rightarrow \{x\}^{*}$ be defined by $\varphi(a) = \varphi(b) = x$, and let \linebreak$\psi : \{a,b\}^{*} \rightarrow \{x\}^{*}$ be defined by $\varphi(a) = x^{3}$, $\psi(b) = x^{2}$. Then $\varphi(\{a,b\}) = \{x\}$ which contains no words with proper suffixes, and hence $Q_{1} = \emptyset$. However, \linebreak $\psi(\{a,b\}) = \{x^{2},x^{3}\}$, for which the set of proper suffixes is $\{x,x^{2}\}$. Hence \linebreak $Q_{2} = \{(\epsc, x), (\epsc, x^{2})\}$. Thus $\mathcal{A}_{\varphi,\psi}$ has state set \[Q = \{\iota,(\epsc,x),(\epsc,x^{2}),(\epsc,\epsc)\}.\]


For the edges in $\delta$, we note that $\Delta_{1} = \Delta_{2} = \emptyset$, as $\epsc \not \in \varphi(\{a,b\})$ and  $\epsc \not \in \psi(\{a,b\})$. $\Delta_{3}$ is also empty, as $\varphi(\{a,b\}) = \{x\}$, for which the only prefixes  are $\epsc$ and $x$, which are not in $\psi(\{a,b\})$. \\

 For $\Delta_{4}$ however, we obtain the edges $(\iota,(a,a),(\epsc,x))$ (as $\varphi(a) = x, \psi(a) = x^{2}$, and $\varphi(a) \leq_{p}\psi(a)$ with $\varphi(a)^{-1}\psi(a) = x^{-1}x^{2} = x$) and $(\iota,(a,b),(\epsc,x^{2}))$ (as $\varphi(a) = x,\linebreak  \psi(b) = x^{3}$, and $\varphi(a) \leq_{p}\psi(b)$ with $\varphi(a)^{-1}\psi(b) = x^{-1}x^{3} = x^{2}$).


$\Delta_{5}$ and $\Delta_{6}$ are empty once more, as there are no states of the form $(u,\epsc) \in Q_{1}$ as $Q_{1}$ is empty. A verification similar to previous calculations however gives the edges\begin{align*}((\epsc,x^{2}),(b,\epsb),(\epsc,x)),&&
((\epsc,x^{2}),(a,\epsb),(\epsc,x)),\\
 ((\epsc,x),(b,\epsb),(\epsc,\epsc)),&&
((\epsc,x),(a,\epsb),(\epsc,\epsc))\end{align*}
from $\Delta_{7}$. Noting that $\Delta_{8} = \emptyset$, we obtain the full automaton $\mathcal{A}_{\varphi,\psi}$, as seen below.

%

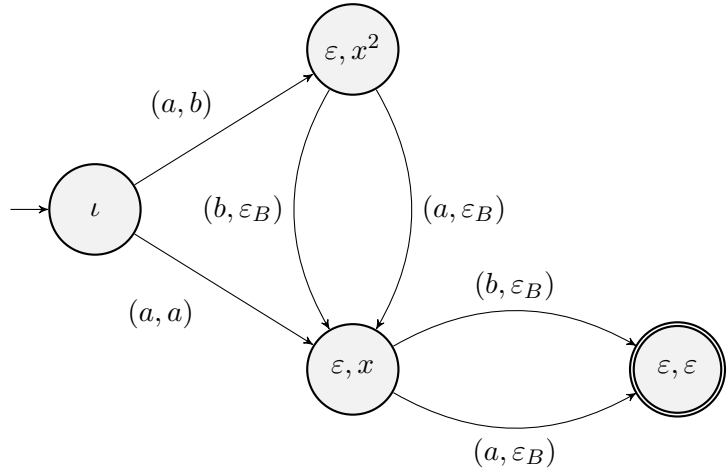
\begin{center}\begin{tikzpicture}
\node[state, initial] (iota) {$\iota$};
\node[state, above right of = iota,xshift=0.5in] (q2) {$\varepsilon,x^{2}$};
\node[state, below right of=iota,xshift=0.5in] (q3) {$\varepsilon,x$};
\node[state, accepting, right of = q3,xshift=0.5in] (q4) {$\varepsilon,\varepsilon$};

\draw 
(iota) edge[above left] node{$(a,b)$} (q2)
(iota) edge[below left] node{$(a,a)\hspace*{0.1in}$} (q3)
(q2) edge[bend left, right] node{$(a,\epsb)$} (q3)
(q2) edge[bend right, left=0.3] node{$(b,\epsb)$} (q3)
(q3) edge[bend right, below] node{$(a,\epsb)$} (q4)
(q3) edge[bend left, above] node{$(b,\epsb)$} (q4);
\end{tikzpicture}\end{center}

\end{exam}

\begin{exam} 
Let $\varphi: \{a,b,c\}^{*} \rightarrow \{x,y\}^{*}$, $\psi : \{a,b,c\}^{*} \rightarrow \{x,y\}^{*}$ be defined by $\varphi(a) = x, \varphi(b) = y, \varphi(c) = xy$ and $\psi(a) = x, \psi(b) = y, \psi(c) = y^{2}$. Then $\mathcal{A}'$ is given below.
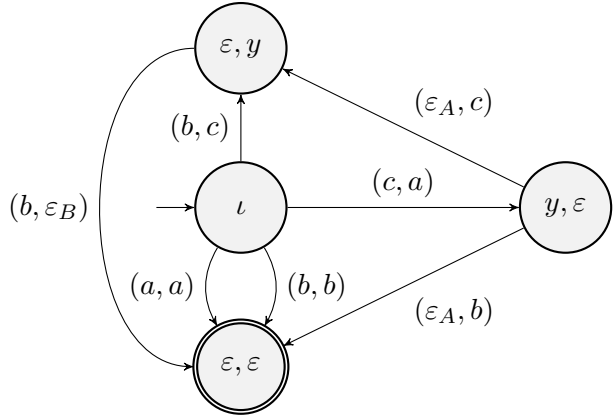
\begin{center}
\begin{tikzpicture}
\node[state, initial] (iota) {$\iota$};
\node[state, accepting, below of = iota, yshift = 0.35in] (end) {$\varepsilon,\varepsilon$};
\node[state, above of = iota, yshift = -0.35in] (q2) {$\varepsilon,y$};
\node[state,  right of = iota,xshift = 0.5in] (q3) {$y,\varepsilon$};

\draw 
(iota) edge[left] node{$(b,c)$} (q2)
(iota) edge[above] node{$(c,a)$} (q3)
(iota) edge[bend right, left] node{$(a,a)$} (end)
(q3) edge[below right] node{$(\epsa,b)$} (end)
(q3) edge[above right] node{$(\epsa,c)$} (q2)
(q2) edge[bend right=90, left] node{$(b,\epsb)$} (end)
(iota) edge[bend left, right] node{$(b,b)$} (end);
\end{tikzpicture}
\end{center}
\end{exam}

We utilise this automatic construction in the following result.

\begin{thm} \thlabel{fgnocyclesaut}
Let $\varphi : A^{*} \rightarrow C^{*}$, $\psi : B^{*} \rightarrow C^{*}$ be two epimorphisms with $A,B,C$ finite alphabets, and let  $\mathcal{A_{\varphi,\psi}}$  be the associated automaton given as above. Then the fiber product of $A^{*}$ with $B^{*}$ over $C^{*}$ with respect to $\varphi, \psi$ is finitely generated if and only if $\mathcal{A_{\varphi,\psi}}$ has no cycles.
\end{thm}

In order to prove this result, we utilise the following lemmas.

\begin{lemma} \thlabel{pathlabelrelation} Let $(u,v) \in Q$, and let $(\alpha,\beta) \in A^{*}\times B^{*}$. If a path from $\iota$ to $(u,v)$ has label $(\alpha,\beta)$, then \begin{equation}\label{pathrelationeqn}\varphi(\alpha)v = \psi(\beta)u.\end{equation}
\begin{proof} We proceed by induction on path length. The paths of length one are precisely the transitions $p \in \Delta_{1}\cup\Delta_{2}\cup\Delta_{3}\cup\Delta_{4}$. For $p \in \Delta_{1}\cup \Delta_{2}$, \eqref{pathrelationeqn} holds by definition . If $p \in \Delta_{3}$, then $p = (\iota,(a,b),(\psi(b)^{-1}\varphi(a),\epsc))$ for some $a \in A, b \in B$, and \[\varphi(\alpha)v = \varphi(a) = \psi(b)\psi(b)^{-1}\varphi(a) = \psi(\beta)u\] as required. Similarly if $p \in \Delta_{4}$, then $p = (\iota,(a,b),(\epsc,\varphi(a)^{-1}\psi(b)))$ for some \linebreak $a \in A, b \in B$, and \[\varphi(\alpha)v = \varphi(a)\varphi(a)^{-1}\psi(b) = \psi(b) = \psi(\beta)u.\] As these are all the paths of length one, this proves the base case. For the inductive hypothesis, assume that if a path from $\iota$ to $(u,v)$ of length $k$ has label $(\alpha,\beta)$, then $\varphi(\alpha)v = \psi(\beta)u$.

Consider a path from $\iota$ to $(u',v')$ of length $k+1$ with label $(\alpha,\beta)$. Then necessarily there exists a path $p$ of length $k$ from $\iota$ to some state $(u,v)$, and a transition $p'$ from $(u,v)$ to $(u',v')$. Then there are two cases for the label of $p'$:

\textit{Case 1:} If $p'$ has label $(\epsa,b)$ for some $b \in B$, then it follows that $v = \epsc$, and the path $p$ has label $(\alpha, \beta b^{-1})$. If $\psi(b) \leq_{p} u$, then $(u',v') = (\psi(b)^{-1}u,\epsc)$ by the definition of $\Delta_{5}$. Hence
\begin{align*}
\psi(\beta)u' & = \psi(\beta b^{-1})\psi(b) u' \\
                              & = \psi(\beta b^{-1})u \\
		         & = \varphi(\alpha)v \text{ by the inductive hypothesis}\\
		         & = \varphi(\alpha)v'. \\
\end{align*}
Otherwise, if $u \leq_{p} \psi(b)$, then $(u',v') = (\epsc, u^{-1}\psi(b))$ by the definition of $\Delta_{6}$ . Hence
\begin{align*}
\psi(\beta)u' & = \psi(\beta b^{-1})\psi(b) \\
                              & = \psi(\beta b^{-1})uv' \\
		         & = \varphi(\alpha)vv' \text{ by the inductive hypothesis}\\
		         & = \varphi(\alpha)v'. \\
\end{align*}

\textit{Case 2:} If $p'$ has label $(a, \epsb)$ for some $a \in A$, then it follows that $u = \epsc$, and the path $p$ has label $(\alpha a^{-1}, \beta)$. If $ \varphi(a) \leq_{p} v$, then $(u',v') = (\epsc,\varphi(a)^{-1}v)$ by the definition of $\Delta_{7}$. Hence
\begin{align*}
\varphi(\alpha)v' & = \varphi(\alpha a^{-1})\varphi(a)v' \\
                              & = \varphi(\alpha a^{-1})v  \\
		         & = \psi(\beta)u \text{ by the inductive hypothesis}\\
                              & = \psi(\beta)u'.
\end{align*}

Otherwise, if $v \leq_{p} \varphi(a)$, then $(u',v') = (v^{-1}\varphi(a),\epsc)$ by the definition of $\Delta_{8}$. Hence
\begin{align*}
\varphi(\alpha)v' & = \varphi(\alpha a^{-1})\varphi(a)\\
                              & = \varphi(\alpha a^{-1})vu' \\
		         & = \psi(\beta)uu' \text{ by the inductive hypothesis}\\
                              & = \psi(\beta)u'.
\end{align*}
Thus the result holds by induction on paths of arbitrary length.\end{proof}

\begin{lemma} \thlabel{onepathlemma} Let $(\alpha,\beta) \in A^{*}\times B^{*}$. Then there is at most one path originating from $\iota$ with label $(\alpha,\beta)$ in $\mathcal{A_{\varphi,\psi}}$.
\begin{proof}
By definition of $\delta$, the only paths originating from $\iota$ with label $(\alpha,\beta)$ for either $\alpha = \epsa$ or $\beta = \epsb$ are the length one transitions $p \in \Delta_{1}\cup\Delta_{2}$, each of which is distinct. \\

 Otherwise, any path $p$ originating from $\iota$ with label $(\alpha,\beta) \in A^{+} \times B^{+}$ is given by a sequence of transitions $(q_{i-1},\sigma_{i},q_{i})_{i=1}^{k}$ such that $q_{0} = \iota$ and $\sigma_{1}\sigma_{2}\dots \sigma_{k} = (\alpha,\beta)$. As $\alpha \in A^{+}$ and $\beta \in B^{+}$, then there exist unique decompositions $\alpha = a_{1}a_{2}\dots a_{|\alpha|}$ for some \linebreak $a_{1},\dots,a_{|\alpha|} \in A$ and $\beta = b_{1}b_{2}\dots b_{|\beta|}$ for some $b_{1},\dots,b_{|\beta|} \in B$.

We claim that state $q_{i-1}$ uniquely determines $\sigma_{i}$ for $1 \leq i \leq k$. As $p$ is a path in $\mathcal{A_{\varphi,\psi}}$, then $q_{i-1} \in \{\iota\}\cup Q_{1} \cup Q_{2}$. By the definition of $\delta$, the only instance where $q_{i-1} = \iota$ is when $i = 1$. Moreover, this implies that $\sigma_{1} = (a_{1},b_{1})$, as $\sigma_{1} \in A \times B$ and $\sigma_{1}\sigma_{2}\dots \sigma_{k} = (\alpha,\beta)$.\\
Further, as $q_{i-1} \in Q_{2}$ if and only if $\sigma_{i} \in A \times \{\epsb\}$, and as $\alpha$ has a unique decomposition over $A$, then $\sigma_{i} \in A \times \{\epsb\}$ if and only if $\sigma_{i} = (a_{j_{i}},\epsb)$ where $j_{i} = |\pi_{A^{*}}(\sigma_{1}\dots\sigma_{i-1})| + 1$.\\
Finally as $\beta$ has a unique decomposition over $B$, a similar proof shows \linebreak $q_{i-1} \in Q_{1} \Leftrightarrow \sigma_{i} = (\epsa, b_{k_{i}})$ where $k_{i} = |\pi_{B^{*}}(\sigma_{1}\dots\sigma_{i-1})| + 1$.

As $Q_{1}\cap Q_{2}\cap\{\iota\} = \emptyset$, then $q_{i-1}$ uniquely determines $\sigma_{i}$, which together uniquely determine $q_{i}$ by the definition of $\delta$. Hence $q_{0} = \iota$ and $(\alpha,\beta) = (a_{1}\dots a_{|\alpha|},b_{1}\dots b_{|\beta|})$ uniquely determine the path $p$.
\end{proof}
\end{lemma}

\end{lemma}

\begin{lemma}\thlabel{quotientstateslemma}
Let $\varphi : A^{*} \rightarrow C^{*}$, $\psi : B^{*} \rightarrow C^{*}$ be two epimorphisms with $A,B,C$ finite alphabets, let $\mathcal{A_{\varphi,\psi}}$  be the associated automaton given as above. Let $\Phi := \varphi \circ \pi_{A^{*}}$, and $\Psi := \psi \circ \pi_{B^{*}}$. Let $(q_{i-1},\sigma_{i},q_{i})_{i=1}^{k}$ be a sequence of transitions with $q_{0} = \iota$. Then either $$q_{k} =(\Psi(\sigma_{1}\dots\sigma_{k})^{-1}\Phi(\sigma_{1}\dots\sigma_{k}),\epsc),$$ if $\Psi(\sigma_{1}\dots\sigma_{k}) \leq_{p} \Phi(\sigma_{1}\dots\sigma_{k})$, or $$q_{k} =(\epsc,\Phi(\sigma_{1}\dots\sigma_{k})^{-1}\Psi(\sigma_{1}\dots\sigma_{k}))$$
if $\Phi(\sigma_{1}\dots\sigma_{k}) \leq_{p} \Psi(\sigma_{1}\dots\sigma_{k})$.
\begin{proof}
We proceed by induction on $k$. Firstly for the base case where $k=1$, as $q_{0} = \iota$, it follows that $(q_{0},\sigma_{1},q_{1}) \in \Delta_{1}\cup\Delta_{2}\cup\Delta_{3}\cup\Delta_{4}$, from which the required form of $q_{1}$ follows by definition. 

%

For the inductive hypothesis, assume that for $k = j$, we have either \begin{equation}q_{j} =(\Psi(\sigma_{1}\dots\sigma_{j})^{-1}\Phi(\sigma_{1}\dots\sigma_{j}),\epsc), \label{statecase1}\end{equation} or \begin{equation}q_{j} =(\epsc,\Phi(\sigma_{1}\dots\sigma_{j})^{-1}\Psi(\sigma_{1}\dots\sigma_{j})),\label{statecase2}\end{equation}
and consider the state $q_{j+1}$ in the case where $q_{j} \not = (\epsc,\epsc)$ . It suffices to assume only case \eqref{statecase1}, as the proof for case \eqref{statecase2} will follow by a symmetric argument. As $\Psi(\sigma_{1}\dots\sigma_{j})^{-1}\Phi(\sigma_{1}\dots\sigma_{j}) \not = \epsc$, by definition of $\Delta_{5},\Delta_{6}$, it follows that $\sigma_{j+1} = (\epsa,b)$ for some $b \in B$, and either 
\begin{align*}
q_{j+1} &= (\psi(b)^{-1}\Psi(\sigma_{1}\dots\sigma_{j})^{-1}\Phi(\sigma_{1}\dots\sigma_{j}),\epsc) \\ 
              &= ([\Psi(\sigma_{1}\dots\sigma_{j})\psi(b)]^{-1}\Phi(\sigma_{1}\dots\sigma_{j}),\epsc) \\
              &= ([\Psi(\sigma_{1}\dots\sigma_{j})\Psi(\sigma_{j+1})]^{-1}\Phi(\sigma_{1}\dots\sigma_{j}),\epsc) \\
              &= (\Psi(\sigma_{1}\dots\sigma_{j+1})^{-1}\Phi(\sigma_{1}\dots\sigma_{j+1}),\epsc),
\end{align*}or \begin{align*}
q_{j+1}  &= (\epsc, [\Psi(\sigma_{1}\dots\sigma_{j})^{-1}\Phi(\sigma_{1}\dots\sigma_{j})]^{-1}\psi(b)) \\ 
               &= (\epsc, \Phi(\sigma_{1}\dots\sigma_{j})^{-1}\Psi(\sigma_{1}\dots\sigma_{j})\psi(b)) \\ 
               &= (\epsc, \Phi(\sigma_{1}\dots\sigma_{j})^{-1}\Psi(\sigma_{1}\dots\sigma_{j})\Psi(\sigma_{j+1})) \\ 
               &= (\epsc, \Phi(\sigma_{1}\dots\sigma_{j+1})^{-1}\Psi(\sigma_{1}\dots\sigma_{j+1}))
\end{align*}as expected. Hence the result follows by induction.\end{proof}\end{lemma}

\begin{lemma} \thlabel{indecomposableL(A)} Let $\varphi : A^{*} \rightarrow C^{*}$, $\psi : B^{*} \rightarrow C^{*}$ be two epimorphisms with $A,B,C$ finite alphabets, and let  $\mathcal{A_{\varphi,\psi}}$  be the associated automaton given as above. Then the language accepted by $\mathcal{A_{\varphi,\psi}}$ is the set of indecomposable elements of $\fibpro$.
\begin{proof} We first show that elements of $\mathcal{L}(\mathcal{A_{\varphi,\psi}})$ are indecomposable in $\fibpro$. Let $(\alpha,\beta) \in \mathcal{L}(\mathcal{A_{\varphi,\psi}})$. Then there is a path $p = (q_{i-1},\sigma_{i},q_{i})_{i=1}^{k}$ from $\iota$ to $(\epsc,\epsc)$ with label $(\alpha,\beta)$. By \thref{onepathlemma}, it follows that $\varphi(\alpha) = \psi(\beta)$, and hence $(\alpha,\beta) \in \fibpro$. 

Further, suppose for a contradiction that $(\alpha,\beta)$ is decomposable. Then \begin{equation}\label{decompalph} (\alpha,\beta) = (\alpha',\beta')(\alpha'',\beta'')\end{equation} for some $\alpha',\alpha'' \in A^{*}$, $\beta',\beta'' \in B^{*}$. To avoid contradiction, it must be that $\alpha \in A^{+}$ and $\beta \in B^{+}$, as the definition of $\delta'$ gives that the only pairs accepted by $\mathcal{A_{\varphi,\psi}}'$ involving $\epsa$ or $\epsb$ are those of the form $(\epsa, b), (a,\epsb)$ for $a \in A, b \in B$ with $\varphi(a) = \psi(b) = \epsc$ which are indecomposable in $\fibpro$. 

By the definition of $\delta'$, either $\sigma_{1} = (\alpha_{1},\epsb)$, $\sigma_{1} = (\epsa,\beta_{1})$ or $\sigma_{1} = (\alpha_{1},\beta_{1})$. The first two possibilities imply that $q_{1} = (\epsc,\epsc)$, which is a contradiction as then either $(\alpha,\beta) = (\alpha_{1},\epsc)$ or $(\epsc,\beta_{1})$, both of which are indecomposable in $\fibpro$. Hence $\sigma_{1} = (\alpha_{1},\beta_{1})$.

Returning to \eqref{decompalph}, note that if $\alpha' = \epsa$, then $\beta' \in B^{+}$ with $\psi(\beta') = \epsc$, implying $\psi(\beta'_{1}) = \psi(\beta_{1}) = \epsc$ also. This is a contradiction, as transitions of the form $(\iota, (\alpha_{1},\beta_{1}), q)$ with $\psi(\beta_{1}) = \epsc$ are excluded from $\delta'$. Similarly, $\beta' = \epsb$ also leads to a contradiction. Hence $\alpha' \in A^{+}$ and $\beta' \in B^{+}$. 

Writing $(\alpha,\beta) = (a_{1}a_{2}\dots a_{|\alpha|},b_{1}b_{2}\dots b_{|\beta|})$, then \begin{equation}\label{fullwordseq} (\alpha,\beta) = (a_{1}\dots a_{m},b_{1}\dots b_{n})(a_{m+1}\dots a_{|\alpha|},b_{n+1}\dots b_{|\beta|})\end{equation}for some $1 \leq m < |\alpha|$, $1 \leq n <|\beta|$. In particular, as $a_{1}\dots a_{m}\leq_{p} \alpha$, $b_{1} \dots b_{n} \leq_{p} \beta$, then there exist minimal $M,N < |\alpha|+|\beta|$ such that $a_{1}\dots a_{m} = \pi_{A^{*}}(\sigma_{1}\dots\sigma_{M})$ and $b_{1}\dots b_{n} = \pi_{B^{*}}(\sigma_{1}\dots\sigma_{N})$. Taking $k= \text{min}\{M,N\}$, it follows that
\[\sigma_{1}\dots \sigma_{k} = \begin{cases} (a_{1}\dots a_{m},b_{1}\dots b_{n'}) \text{ for some } n' < n & \text{ if } k = M \\                        (a_{1}\dots a_{m'},b_{1}\dots b_{n}) \text{ for some } m' < m & \text{ if } k = N.\end{cases}\] 

If $k = M$, then as $\psi(b_{1}\dots b_{l}) \leq_{p} \psi(b_{1}\dots b_{n})$ for all $n' \leq l \leq n$  and \linebreak$\psi(b_{1}\dots b_{n}) = \varphi(a_{1}\dots a_{m})$, by \thref{quotientstateslemma} it follows that $q_{M+t} \in Q_{1}$ and hence \linebreak$\sigma_{M+t+1} = b_{n'+t+1}$ for $0 \leq t < n-n'$. Thus $\sigma_{1}\dots\sigma_{M+n-n'} = (a_{1}\dots a_{m},b_{1}\dots b_{n})$, and thus $q_{M+n-n'} = (\epsc,\epsc)$ by \thref{quotientstateslemma}. But $M + n-n' = m+ n< |\alpha| + |\beta|$, which contradicts acceptance of $(\alpha,\beta)$ by $p$ (as there are no out-edges from $(\epsc,\epsc)$).

A similar proof shows if $k = N$, then $\sigma_{1}\dots\sigma_{N+m-m'} = (a_{1}\dots a_{m},b_{1}\dots b_{n})$, and $q_{N+m-m'} = (\epsc,\epsc)$. But $N +m-m' = m + n < |\alpha| + |\beta|$, which again contradicts acceptance of $(\alpha,\beta)$ by $p$. Thus is must be that $(\alpha,\beta)$ is indecomposable, and hence $\mathcal{L}(\mathcal{A_{\varphi,\psi}})$ consists of indecomposables.\\

To show the reverse inclusion, let $(\alpha,\beta) \in \fibpro$ be indecomposable. If $\alpha = \epsc$, then necessarily $\beta \in B$ and $\psi(\beta) = \epsc$, and the transition $(\iota, (\epsa,\beta),(\epsc,\epsc))$ accepts $(\alpha,\beta)$. Similarly, if $\beta = \epsc$, then $\alpha \in A$ with $\varphi(\alpha) = \epsc$, and the transition $(\iota,(\alpha,\epsb),(\epsc,\epsc))$ accepts $(\alpha,\beta)$. Otherwise, for $(\alpha,\beta) \in A^{+} \times B^{+}$, define the sequence of triples $(q_{i-1},\sigma_{i},q_{i})_{i=1}^{|\alpha| + |\beta|-1} \in Q \times \Sigma \times Q$ by $q_{0} = \iota$, $\sigma_{1} = (a_{1},b_{1})$, and
\begin{equation}\label{fullwordssequ}\sigma_{i} = \begin{cases}
				    (a_{j_{i-1}},\epsb) & \text{ if } q_{i-1} \in Q_{2} \\
				    (\epsa, b_{k_{i-1}}) & \text{ if } q_{i-1} \in Q_{1} \end{cases}\end{equation}
(where $j_{i-1} = |\pi_{A^{*}}(\sigma_{1}\dots\sigma_{i-1})|+1$, $k_{i-1} =|\pi_{B^{*}}(\sigma_{1}\dots \sigma_{i-1})|+1$ for $2 \leq i \leq |\alpha|+|\beta|-1$), and
\[q_{i} = \begin{cases} 
(\Psi(\sigma_{1}\dots\sigma_{i})^{-1}\Phi(\sigma_{1}\dots\sigma_{i}),\epsc) & \text{ if }  \Psi(\sigma_{1}\dots\sigma_{i}) \leq_{p} \Phi(\sigma_{1}\dots\sigma_{i}) \\
(\epsc,\Phi(\sigma_{1}\dots\sigma_{i})^{-1}\Psi(\sigma_{1}\dots\sigma_{i})) & \text{ if } \Phi(\sigma_{1}\dots\sigma_{i}) \leq_{p} \Psi(\sigma_{1}\dots\sigma_{i})\end{cases}\]
for $1 \leq i \leq |\alpha| + |\beta|-1$. 

Note that both $q_{i}$ and $\sigma_{i}$ are always well defined, as if $\varphi(\alpha) = \psi(\beta)$, then either $\Phi(\sigma_{1}\dots\sigma_{i}) \leq_{p} \Psi (\sigma_{1}\dots\sigma_{i})$ or $\Psi(\sigma_{1}\dots\sigma_{i}) \leq_{p} \Phi (\sigma_{1}\dots\sigma_{i})$, as $\Phi(\sigma_{1}\dots\sigma_{i})$ and $\Psi(\sigma_{1}\dots\sigma_{i})$ are prefixes of $\varphi(\alpha)$ and $\psi(\beta)$ respectively. Moreover, $\Phi(\sigma_{1}\dots\sigma_{i}) \not =\Psi(\sigma_{1}\dots\sigma_{i})$ for \linebreak $1 \leq i \leq |\alpha| + |\beta| -1$ by indecomposability of $(\alpha,\beta)$, and hence $q_{i} \not = (\epsc,\epsc)$ for $0 \leq i \leq |\alpha| + |\beta|-1$.  By construction of $\sigma_{i}$, noting that if $\sigma_{i} \in A\times\{\epsb\}$, then
\[j_{i}  = |\pi_{A^{*}}(\sigma_{1}\dots\sigma_{i})| + 1  = ( |\pi_{A^{*}}(\sigma_{1}\dots\sigma_{i-1})| +1) + 1=  j_{i-1}+1,\]
and
\[k_{i}  = |\pi_{B^{*}}(\sigma_{1}\dots\sigma_{i})| + 1 =|\pi_{B^{*}}(\sigma_{1}\dots\sigma_{i-1})| + 1 =  k_{i-1},\]
whereas if $\sigma_{i} \in \{\epsa\}\times B$, then
\[j_{i}  = |\pi_{A^{*}}(\sigma_{1}\dots\sigma_{i})| + 1  = ( |\pi_{A^{*}}(\sigma_{1}\dots\sigma_{i-1})| + 1=  j_{i-1},\]
and
\[k_{i}  = |\pi_{B^{*}}(\sigma_{1}\dots\sigma_{i})| + 1 =(|\pi_{B^{*}}(\sigma_{1}\dots\sigma_{i-1})| +1) + 1 =  k_{i-1}+1.\]
As $j_{1} = k_{1} = 2$, then it follows that $\pi_{A^{*}}(\sigma_{1}\dots\sigma_{|\alpha|+|\beta|}) = a_{1}a_{2}\dots a_{m}$ and \linebreak $\pi_{B^{*}}(\sigma_{1}\dots\sigma_{|\alpha|+|\beta|}) = b_{1}b_{2}\dots b_{n}$ for some $m \leq |\alpha|$, $n \leq |\beta|$.  We conclude by making the following claims;

\textit{Claim 1: $(q_{i-1},\sigma_{i},q_{i}) \in \delta$ for $1 \leq i \leq |\alpha| + |\beta|$, hence $(q_{i-1},\sigma_{i},q_{i})_{i=1}^{|\alpha| + |\beta|}$ is a path in $\mathcal{A_{\varphi,\psi}}$}. \\
\textit{Claim 2: $\sigma_{1}\dots\sigma_{|\alpha|+|\beta|} = (\alpha,\beta)$}. 

Combining the above claims, and noting that $q_{0} = \iota$, $q_{|\alpha|+|\beta|} = (\epsc,\epsc)$ as $\varphi(\alpha) = \psi(\beta)$ and hence
\[\Psi(\sigma_{1}\dots\sigma_{|\alpha|+|\beta|})^{-1}\Phi(\sigma_{1}\dots\sigma_{|\alpha|+|\beta|}) = \Phi(\sigma_{1}\dots\sigma_{|\alpha|+|\beta|})^{-1}\Psi(\sigma_{1}\dots\sigma_{|\alpha|+|\beta|}) = \epsc,\]
then there exists a path in $\mathcal{A_{\varphi,\psi}}$ accepting $(\alpha,\beta)$, and hence $(\alpha,\beta) \in \mathcal{L}(\mathcal{A_{\varphi,\psi}})$, completing the proof of the theorem.

\textit{Proof of Claim 1.} For $i = 1$, $(q_{0},\sigma_{1},q_{1}) \in \delta$ by the definition, as \linebreak necessarily $\sigma_{1} = (a_{1},b_{1})$ and either $q_{1} = (\Psi(\sigma_{1})^{-1}\Phi(\sigma_{1}),\epsc) = (\psi(b_{1})^{-1}\varphi(a_{1}),\epsc)$, or $q_{1} = (\epsc,\Phi(\sigma_{1})^{-1}\Psi(\sigma_{1})) = (\epsc,\varphi(a_{1})^{-1}\psi(b_{1}))$. 

If $q_{i-1} \in Q_{1},$ then $q_{i-1} = (\Psi(\sigma_{1}\dots\sigma_{i-1})^{-1}\Phi(\sigma_{1}\dots\sigma_{i-1}),\epsc)$ and $\sigma_{i} = (\epsa,b_{k_{i}})$. As $\Psi(\sigma_{1}\dots \sigma_{i})$ and $\Phi(\sigma_{1}\dots \sigma_{i})$ are prefixes of $\psi(\beta) = \varphi(\alpha)$, then either \linebreak $\Psi(\sigma_{1}\dots \sigma_{i}) \leq_{p} \Phi (\sigma_{1}\dots\sigma_{i})$ or $\Phi (\sigma_{1}\dots\sigma_{i}) \leq_{p}\Psi(\sigma_{1}\dots \sigma_{i})$. \\

If $\Psi(\sigma_{1}\dots \sigma_{i}) \leq_{p} \Phi (\sigma_{1}\dots\sigma_{i})$, then in particular it follows that \linebreak$\psi(b_{k_{i}}) = \Psi(\sigma_{i}) \leq_{p} \Psi(\sigma_{1}\dots\sigma_{i-1})^{-1}\Phi(\sigma_{1}\dots\sigma_{i-1})$, and
\begin{align*}
q_{i} & = (\Psi(\sigma_{1}\dots\sigma_{i})^{-1}\Phi(\sigma_{1}\dots\sigma_{i}),\epsc)\\
         & = ([\Psi(\sigma_{1}\dots\sigma_{i-1})\Psi(\sigma_{i})]^{-1}\Phi(\sigma_{1}\dots\sigma_{i-1}),\epsc)\\
         & = (\Psi(\sigma_{i})^{-1}\Psi(\sigma_{1}\dots\sigma_{i-1})^{-1}\Phi(\sigma_{1}\dots\sigma_{i-1}),\epsc)\\
         & = (\psi(b_{k_{i}})^{-1}\Psi(\sigma_{1}\dots\sigma_{i-1})^{-1}\Phi(\sigma_{1}\dots\sigma_{i-1}),\epsc),
\end{align*}

thus $(q_{i-1},\sigma_{i},q_{i}) \in \Delta_{5} \subseteq \delta $ by definition. On the other hand, if \linebreak $\Phi(\sigma_{1}\dots\sigma_{i}) \leq_{p} \Psi(\sigma_{1}\dots\sigma_{i})$, then in particular it follows that \linebreak $ \Psi(\sigma_{1}\dots\sigma_{i-1})^{-1}\Phi(\sigma_{1}\dots\sigma_{i-1}) \leq_{p} \psi(b_{k_{i}})$, and
\begin{align*}
q_{i} & = (\epsc,\Phi(\sigma_{1}\dots\sigma_{i})^{-1}\Psi(\sigma_{1}\dots\sigma_{i}))\\
         & = (\epsc,\Phi(\sigma_{1}\dots\sigma_{i-1})^{-1}\Psi(\sigma_{1}\dots\sigma_{i-1})\Psi(\sigma_{i}))\\
         & = (\epsc,[\Psi(\sigma_{1}\dots\sigma_{i-1})^{-1}\Phi(\sigma_{1}\dots\sigma_{i-1})]^{-1}\Psi(\sigma_{i}))\\
         & = (\epsc,[\Psi(\sigma_{1}\dots\sigma_{i-1})^{-1}\Phi(\sigma_{1}\dots\sigma_{i-1})]^{-1}\psi(b_{k_{i}})),
\end{align*}

thus $(q_{i-1},\sigma_{i},q_{i}) \in \Delta_{6} \subseteq \delta $ by definition. \\

If $q_{i-1} \in Q_{2}$, then a similar proof shows that $(q_{i-1},\sigma_{i},q_{i}) \in \Delta_{6}\cup\Delta_{8} \subseteq \delta$, thus proving the claim.

\textit{Proof of Claim 2.} It suffices to prove that $m = |\alpha|$, and $n = |\beta|$. Let $S_{A},S_{B}$ be defined by\begin{align*}
S_{A} &:= \{ i \in \{2,\dots,|\alpha|+|\beta|\} : \sigma_{i} \in A \times \{\epsb\}\}\\
S_{B} &:= \{ i \in \{2,\dots,|\alpha|+|\beta|\} : \sigma_{i} \in \{\epsa\} \times B\}
\end{align*} 
Then $m = |S_{A}|+1$ and $n = |S_{B}|+1$. Suppose for a contradiction that $|S_{A}| > |\alpha|-1$. Ordering $S_{A}$ in the natural way, let $I \in S_{A}$ be the element at  position $|\alpha|$. Then \[|\pi_{A^{*}}(\sigma_{1}\dots\sigma_{I-1})| = |\pi_{A^{*}}(\sigma_{2}\dots\sigma_{I-1})| + 1 = |\pi_{A^{*}}(\sigma_{2}\dots\sigma_{I})| -1 + 1 = |\alpha|.\] Moreover, as $\pi_{A^{*}}(\sigma_{1}\dots\sigma_{I-1}) \leq_{p} a_{1}a_{2}\dots a_{m}$, it follows that $\pi_{A^{*}}(\sigma_{1}\dots\sigma_{I-1}) = a_{1}a_{2}\dots a_{|\alpha|}$. In particular, $\Phi(\sigma_{1}\dots\sigma_{I-1}) = \varphi(\alpha)$. But
\begin{align*} &\, \varphi(\alpha) = \psi(\beta) \\
\Rightarrow &\,\Psi(\sigma_{1}\dots\sigma_{I-1}) \leq_{p} \Phi(\sigma_{1}\dots\sigma_{I-1}) \\
\Rightarrow &\, q_{I-1} \in Q_{1}. \end{align*}
As $q_{I-1} \in Q_{1}$, then $\sigma_{I} \in \{\epsa\}\times B$, which is a contradiction, as then $I \not \in S_{A}$. Hence $|S_{A}| \leq |\alpha|-1$. A similar proof shows that $|S_{B}| \leq |\beta|-1$. Moreover, as \linebreak$|S_{A}| + |S_{B}| = |\alpha| -1 + |\beta|-1$, it follows that $|S_{A}| \not < |\alpha|-1$ (for otherwise $|S_{B}| > |\beta|-1$), and similarly $|S_{B}| \not < |\beta|-1$. Hence $|S_{A}| = |\alpha|-1$, $|S_{B}| = |\beta|-1$ which gives $m = |\alpha|$, $n = |\beta|$ as required.\end{proof}\end{lemma}

\begin{proof}[Proof of \thref{fgnocyclesaut}]\label{constructingcyclesproof}
$(\Rightarrow)$ For sufficiency, we prove the contrapositive. Suppose that $\mathcal{A_{\varphi,\psi}}$ has a cycle. Then there exists a sequence of transitions $(q_{i-1},\sigma_{i},q_{i})_{i=1}^{k} $ where $q_{0 } = q_{k}$. By the definition of $\delta$, it follows that $q_{0} \not = \iota$ and $q_{0} \not = (\epsc,\epsc)$. Thus either $q_{0} = (u,\epsc)$ where $u \in C^{+}$ is such that $u <_{s} w$ for some $w \in \varphi(A)$, or $q_{0} = (\epsc,v)$ where $v \in C^{+}$ is such that $v <_{s} w$ for some $w \in \psi(B)$.

It suffices to consider the case where $q_{0} = (u,\epsc)$, as the proof for the other case will follow by a symmetric argument. As $u$ is a suffix of $w = \varphi(a)$ for some $a \in A$, and $\psi$ is surjective, then there exist $b_{1},\dots,b_{j},b'_{1},\dots,b'_{l} \in B$ such that $\psi(b_{1}\dots b_{j})u = w$ and $\psi(b'_{1}\dots b'_{l}) = u$.

Construct the sequences of transitions $(p_{i-1},\tau_{i},p_{i})_{i=1}^{j}$ and $(r_{i-1},\rho_{i},r_{i})_{i=1}^{l}$  where 
\begin{enumerate} \item $(p_{0},\tau_{1},p_{1}) = (\iota, (a,b_{1}),([\psi(b_{1})]^{-1}w,\epsc))$, 
			\item $p_{i} = ([\psi(b_{1}...b_{i})]^{-1}w,\epsc)$, $\tau_{i} = (\epsa,b_{i})$ for $1 < i \leq j$ , 
			\item $r_{0} = q_{0}$, $\rho_{i} = (\epsa, b'_{i})$, $r_{i} = ([\psi(b'_{1}\dots b'_{i})]^{-1}u,\epsc)$ for $1 \leq i \leq l$.\end{enumerate}

Noting that $p_{j} = q_{0} = q_{k} = r_{0}$ and $r_{l} = (\epsc,\epsc)$, then for all $n \in \N$ it follows that the input $\tau_{1}\dots\tau_{j+1}(\sigma_{1}\dots\sigma_{k})^{n}\rho_{1}\dots\rho_{l}$ is accepted by $\mathcal{A_{\varphi,\psi}}$, via the concatenation of the sequences of transitions $(p_{i-1},\tau_{i},p_{i})_{i=1}^{j}$, $(q_{i-1},\sigma_{i},q_{i})_{i=1}^{k}$\ ($n$ times), and $(r_{i-1},\rho_{i},r_{i})_{i=1}^{l}$.

$(\Leftarrow)$ For necessity, suppose that $\mathcal{A_{\varphi,\psi}}$ has no cycles. Then as $|Q|, |\delta| < \infty$, it follows that there are only finitely many transitions in $\mathcal{A_{\varphi,\psi}}$, and hence $|\mathcal{L}(\mathcal{A_{\varphi,\psi}})| < \infty$. By \thref{indecomposableL(A)}, it follows that $\fibpro$ has finitely many indecomposable elements. Hence by \thref{FGiffIndecompFin}, $\fibpro$ is finitely generated as required.
\end{proof}

We also give an analogous result for fiber products of two finitely generated free semigroups over a finitely generated free semigroup fiber. Given epimorphisms $\varphi : A^{+} \rightarrow C^{+}$, \linebreak $\psi : B^{+} \rightarrow C^{+}$ (with $A,B,C$ finite alphabets), we can extend $\varphi$ and $\psi$ naturally to homomorphisms $\varphi' : A^{*} \rightarrow C^{*}$, $\psi' : B^{*} \rightarrow C^{*}$ by mapping $\epsa$ and $\epsb$ to $\epsc$. Then $\Pi(\varphi',\psi') = \fibpro \cup \{(\epsa,\epsb)\}$, and hence $\fibpro$ is finitely generated as a semigroup if and only if $\Pi(\varphi',\psi')$ is finitely generated as a monoid. Hence we obtain the following corollaries:

\begin{cor} \thlabel{indecomposableL(A)} Let $\varphi : A^{*} \rightarrow C^{*}$, $\psi : B^{*} \rightarrow C^{*}$ be two epimorphisms with $A,B,C$ finite alphabets. Then the language accepted by $\mathcal{A_{\varphi',\psi'}}$ is the set of indecomposable elements of $\fibpro$.\end{cor}

\begin{cor} \thlabel{A'nocyclesfg}
Let $\varphi : A^{+} \rightarrow C^{+}$, $\psi : B^{+} \rightarrow C^{+}$ be two epimorphisms with $A,B,C$ finite alphabets. Then the fiber product of $A^{+}$ with $B^{+}$ over $C^{+}$ with respect to $\varphi,\psi$ is finitely generated if and only if $\mathcal{A_{\varphi',\psi'}}$ has no cycles.\end{cor}

\section{Some remarks on numbers of subdirect products}\label{nonfibsec}

Though the results above appear to indicate that finitely generated fiber products of free semigroups are sparse, we know that finitely generated subdirect products of finitely generated free semigroups are easy to come by. For example, let $A,B$ be finite alphabets. Then choosing $X \subseteq A\times B$ such that the natural projection maps $\pi_{A} : X \rightarrow A$ and $\pi_{B} : X \rightarrow B$ are surjections yields subdirect products $\langle X \rangle$ of $A^{+}$ and $B^{+}$. It is then possible to count all such $X$, as in the next result.

\begin{propn}\thlabel{numberoffgsubs}
Let $A,B$ be finite sets, and let
\[\textup{Subdirect}(A,B) := \{X \subseteq A \times B : \langle X \rangle \leq_{\textup{sd}} A^{+} \times B^{+}\}.\]
Then
\begin{equation}\label{subdfinnum}|\textup{Subdirect}(A,B)| = \sum_{i=0}^{|A|} (-1)^{i} {|A|\choose i}(2^{|A|-i}-1)^{|B|}.\end{equation}
Moreover, 
\[\displaystyle \lim_{|A| \rightarrow \infty} \frac{|\textup{Subdirect}(A,A)|}{|\mathcal{P}(A\times A)|} = 1.\]
\begin{proof}
Without loss of generality, as $A$ and $B$ are finite we can relabel $A$ and $B$ so that $A =  \{a_{1},\dots,a_{m}\}$ and $B = \{b_{1},\dots,b_{n}\}$ for some $m,n \in \N$. We can associate any $X \in \text{Subdirect}(A,B)$ to the binary $m$ by $n$ matrix $M_{X}$ defined by
\[(M_{X})_{i,j} = \begin{cases}1 & \text{if } (a_{i},b_{j}) \in X\\ 0 & \text{otherwise}\end{cases}.\]
In particular, as $\langle X \rangle$ is a subdirect product of $A^{+}$ and $B^{+}$, then every $a_{i} \in A$ is paired with at least one $b_{j} \in B$, and vice versa. Hence $M_{X}$ has no zero rows or columns. Conversely, we can identify every binary $m$ by $n$ matrix $M$ with no zero rows or columns to a subset $X_{M} \in \text{Subdirect}(A,B)$, where \[X_{M} := \{(a_{i},b_{j}) \in A \times B : (M)_{i,j} = 1\}. \]
Hence $ |\textup{Subdirect}(A,B)|$ is equal to the number of binary $m$ by $n$ matrices with no zero rows or columns. Thus \eqref{subdfinnum} follows by a standard inclusion exclusion argument.

Moreover, for the limit, as $\textup{Subdirect}(A,A) \subseteq \mathcal{P}(A\times A)$, then
\begin{equation}\label{upperboundsum}\frac{|\textup{Subdirect}(A,A)|}{|\mathcal{P}(A\times A)|} \leq 1.\end{equation}
On the other hand, as
\begin{equation}\label{lowerboundsum} |\textup{Subdirect}(A,A)|= (2^{|A|}-1)^{|A|} - |A|(2^{|A|-1}-1)^{|A|} + \sum_{i=2}^{|A|} (-1)^{i}{|A|\choose i}(2^{|A|-i}-1)^{|A|},\end{equation}
then by verifying that the summand values $x_{i} := {|A|\choose i}(2^{|A|-i}-1)^{|A|}$ form a strictly decreasing sequence $(x_{i})_{i=2}^{|A|}$, we see that
\[\sum_{i=2}^{|A|} (-1)^{i}{|A|\choose i}(2^{|A|-i}-1)^{|A|} \geq 0,\]
implying from \eqref{lowerboundsum} that
\[ |\textup{Subdirect}(A,A)| \geq (2^{|A|}-1)^{|A|} - |A|(2^{|A|-1}-1)^{|A|}.\]
Thus as
\begin{equation*}\left( 1- \frac{1}{2^{|A|}}\right)^{|A|} - \frac{|A|}{2^{|A|}} = \frac{(2^{|A|}-1)^{|A|} - |A|(2^{|A|-1})^{|A|}}{2^{|A|^{2}}} \leq \frac{ (2^{|A|}-1)^{|A|} - |A|(2^{|A|-1}-1)^{|A|}}{2^{|A|^{2}}},\end{equation*}
then 
\begin{equation*}\left( 1- \frac{1}{2^{|A|}}\right)^{|A|} - \frac{|A|}{2^{|A|}} \leq \frac{|\textup{Subdirect}(A,A)|}{|\mathcal{P}(A\times A)|} \leq 1,\end{equation*}
from which the limit follows from standard analytic arguments.
\end{proof}
\end{propn}

\thref{numberoffgsubs} suggests that finitely generated subdirect products are numerous within the class of subsemigroups of $A^{+} \times B^{+}$ generated by subsets of $A \times B$, as $|A|$ grows with $|B|$. It is natural to ask how many of those finitely generated subdirect products are fiber products (using \thref{FleischerLemma}), and what proportion of all such subdirect products they constitute. This is answered in the following result.

\begin{propn}\thlabel{numberoffgfibs}
Let $A,B$ be finite sets, and let\[\textup{Subdirect}(A,B) := \{X \subseteq A \times B : \langle X \rangle \leq_{\textup{sd}} A^{+} \times B^{+}\},\]
\[\textup{Fiber}(A,B) := \{X \subseteq \textup{Subdirect}(A,B) : \textup{ker}\,\pi_{A^{+}} \circ \textup{ker}\,\pi_{B^{+}} = \textup{ker}\,\pi_{B^{+}} \circ \textup{ker}\,\pi_{A^{+}} \}.\]

Then
\begin{equation}\label{fibfinnum}|\textup{Fiber}(A,B)| = \sum_{i=1}^{\textup{min}\{|A|,|B|\}}i! \, S(|A|,i)S(|B|,i),\end{equation}
where $S(n,k)$ is the Stirling number of the second kind.

Moreover, 
\[\displaystyle \lim_{|A| \rightarrow \infty} \frac{|\textup{Fiber}(A,A)|}{|\textup{Subdirect}(A,A)|} = 0.\]
\begin{proof}
We claim that $|\textup{Fiber}(A,B)|$ is equal to the number of binary $m \times n$ matrices with no zero rows, zero columns, or submatrices of the form 
\begin{equation}\label{bannedmats} \left( \begin{matrix*}0 & 1\\ 1 & 1 \end{matrix*}\right), \left( \begin{matrix*}1 & 0\\ 1 & 1 \end{matrix*}\right), \left( \begin{matrix*}1 & 1\\ 0 & 1 \end{matrix*}\right), \left( \begin{matrix*}1 & 1\\ 1 & 0 \end{matrix*}\right). \end{equation} 

To prove this, without loss of generality, let $A = \{a_{1}, \dots, a_{m}\}$, and $B = \{b_{1}, \dots, b_{n}\}$. We can associate an $m \times n$ binary matrix  $M_{X}$ to each $X \in \textup{Fiber}(A,B)$ as in \thref{numberoffgsubs}. In particular, $M_{X}$ has no zero rows or columns.

Any two non-zero entries in the same row of $M_{X}$ correspond to two generating pairs $(a_{i},b_{j}), (a_{i'},b_{j'}) \in \langle X \rangle$ with $i = i'$, and hence correspond to generating pairs which are related by the congruence $\textup{ker}\,\pi_{A^{+}}$. Similarly, any two non-zero entries in the same column of $M_{X}$ correspond to two generating pairs which are related by the congruence $\textup{ker}\,\pi_{B^{+}}$. Hence there are no submatrices of $M_{X}$ of the type given in \eqref{bannedmats} as these correspond to pairs $(a_{i},b_{j}),(a_{i'},b_{j'}) \in X$ which are related by $\text{ker}\,\pi_{A^{+}} \circ \text{ker}\,\pi_{B^{+}}$ but not by \linebreak $\textup{ker}\,\pi_{B^{+}} \circ \textup{ker}\,\pi_{A^{+}}$, or vice versa.

Conversely, let $((u,v),(u',v')) \in \langle X \rangle \times \langle X \rangle $. As \[((u,v),(u',v')) \in \text{ker}\,\pi_{A^{+}} \circ \text{ker}\,\pi_{B^{+}} \Leftrightarrow ((u_{i},v_{i}),(u'_{i},v'_{i})) \in \text{ker}\,\pi_{A^{+}} \circ \text{ker}\,\pi_{B^{+}}\]
for $1 \leq i \leq |u|$, and similarly
\[((u,v),(u',v')) \in \text{ker}\,\pi_{B^{+}} \circ \text{ker}\,\pi_{A^{+}} \Leftrightarrow ((u_{i},v_{i}),(u'_{i},v'_{i})) \in \text{ker}\,\pi_{B^{+}} \circ \text{ker}\,\pi_{A^{+}}\]

for $1 \leq i \leq |u|$, then the congruences $\textup{ker}\,\pi_{A^{+}}$ and $\textup{ker}\,\pi_{B^{+}}$ on $\langle X \rangle$ are completely determined by their restrictions to $X$, and hence so are $\text{ker}\,\pi_{A^{+}} \circ \text{ker}\,\pi_{B^{+}}$ and \linebreak $\text{ker}\,\pi_{B^{+}} \circ \text{ker}\,\pi_{A^{+}}$ . Hence every binary matrix without $2 \times 2$ submatrices of the type given in \eqref{bannedmats} corresponds to a fiber product of $A^{+}$ with $B^{+}$. This proves the claim.

The number of $m \times n$ binary matrices allowing for zero rows and columns without submatrices of the above form has been given in \cite[Theorem 3.1]{sterlingcount} as
\begin{equation*}\sum_{i=0}^{\textup{min}\{m,n\}}i! \, S(m+1,i+1)S(n+1,i+1).\end{equation*}
via transforming each matrix into a block diagonal binary matrix, and associating this matrix with two set partitions $\mu$ and $\nu$ of $\{1,\dots m+1\}$ and $\{1,\dots,n+1\}$ into $i+1$ blocks for some $i \in \N$, and a permutation on $\{1,\dots,i\}$. Noting that matrices with no zero rows or columns that avoid the set of submatrices given in \eqref{bannedmats} can be transformed into block diagonal matrices without any zero blocks on the diagonal, and accounting for this in the proof of \cite[Theorem 3.1]{sterlingcount} gives the result in \eqref{fibfinnum}.

For the limit, as $S(n,k)$ is the number of ways to partition a set of size $n$ into $k$ non-empty blocks, which is less than the number of ways to assign a set of size $n$ objects to $k$ unlabelled bins (allowing for empty bins), then we get the following upper bound on the number of fiber products.

\begin{eqnarray*}
|\textup{Fiber}(A,A)| = \sum_{i=1}^{|A|} i! \, S(|A|,i)^{2} \leq \sum_{i=1}^{|A|} i! \, \left( \frac{i^{|A|}}{i!}\right)^{2} \leq \sum_{i=1}^{|A|} i^{2|A|} \leq |A|.|A|^{2|A|} = |A|^{2|A|+1}.
\end{eqnarray*}

Hence using \thref{numberoffgsubs}, we have

\[\lim_{|A|\rightarrow \infty}  \frac{|\textup{Fiber}(A,A)|}{|\textup{Subdirect}(A,A)|} = \lim_{|A|\rightarrow \infty}  \frac{|\textup{Fiber}(A,A)|}{2^{|A|^{2}}} \leq  \lim_{|A|\rightarrow \infty} \frac{|A|^{2|A|+1}}{2^{|A|^{2}}}  \]
which tends to zero by standard analytic arguments. 
\end{proof}
\end{propn}

\newpage
\section{Further questions}\label{quesec}
Most of the results in this paper indicate that fiber products of free semigroups are rarely finitely generated, though finitely generated subdirect products of free semigroups are abundant. The quotient can be defined for these non-fiber products, but it does not determine the semigroup in the same way as the fiber product. Moreover, the results limit the possible presentations for fiber quotients of finitely generated fiber products of free semigroups. These observations motivate the following open questions:

\begin{que}
Does there exist a finitely generated subdirect product of two free monoids with a finite quotient, which is not finitely presented?
\end{que}

\begin{que}
Given a subdirect product $S$ of two free semigroups by a finite set of generating pairs, is it decidable whether or not $S$ is a fiber product?
\end{que}

\begin{que}
Does \thref{fingenifffinpres} generalise to infinite fiber quotients? That is, is a finitely generated fiber product of two free semigroups with an infinite fiber quotient also finitely presented?
\end{que}

\section{Acknowledgements}
The author would like to particularly thank Professor Nik Ru\v{s}kuc for his ever supportive guidance during the writing of this paper. For my aunt, Jacky.

\end{document}